\documentclass[11pt]{amsart}
\usepackage{amssymb}
\usepackage{amsmath}
\usepackage{amsthm}
\usepackage[textwidth=16cm, textheight=25cm]{geometry}
\usepackage{a4wide}

\usepackage[english]{babel}
\usepackage[utf8]{inputenc}
\usepackage[T1]{fontenc}
\usepackage{todonotes}
\usepackage{xcolor}
\makeatletter
\protected\def\ccell#1#{%
  \kern-\fboxsep
  \@ccell{#1}%
}
\def\@ccell#1#2#3{%
  \colorbox#1{#2}{#3}%
  \kern-\fboxsep
}
\makeatother
\usepackage{tikz-cd}
\usetikzlibrary{calc}
\usepackage{enumitem}
\usepackage{booktabs}

\usepackage[pdfborder={0 0 0}]{hyperref}
\hypersetup{
    colorlinks,
    linkcolor={red!80!black},
    citecolor={blue!80!black},
    urlcolor={blue!80!black}
}
\usepackage{setspace}
\usepackage{ mathrsfs } 
\usepackage{multicol}
\numberwithin{equation}{section}
\numberwithin{table}{section}
\makeatletter
\let\c@equation\c@table
\makeatother
\newtheorem{theorem}{Theorem}[section]
\newtheorem{corollary}[theorem]{Corollary}
\newtheorem{lemma}[theorem]{Lemma}

\newtheorem{proposition}[theorem]{Proposition}
\newtheorem{question}[theorem]{Question}
\theoremstyle{definition}
\newtheorem{definition}[theorem]{Definition}
\newtheorem{assumption}[theorem]{Assumption}
\newtheorem{remark}[theorem]{Remark}
\newtheorem{example}[theorem]{Example}

\DeclareMathOperator{\id}{id}
\DeclareMathOperator{\pr}{pr}
\DeclareMathOperator{\Aut}{Aut}
\DeclareMathOperator{\coker}{coker}
\DeclareMathOperator{\im}{im}
\DeclareMathOperator{\Spec}{Spec}

\DeclareMathOperator{\Hom}{Hom}
\DeclareMathOperator{\Ext}{Ext}
\DeclareMathOperator{\OHilb}{Hilb}

\DeclareMathOperator{\GL}{GL}
\DeclareMathOperator{\SL}{SL}

\DeclareMathOperator{\reg}{reg}

\DeclareMathOperator{\depth}{depth}
\DeclareMathOperator{\tanspace}{Tan}
\DeclareMathOperator{\charr}{char}
\DeclareMathOperator{\Sym}{S}%
\DeclareMathOperator{\DP}{\Gamma}%
\DeclareMathOperator{\Sch}{\mathbb{S}}%

\newcommand{\tensor}{\otimes}

\newcommand{\into}{\hookrightarrow}
\newcommand{\hook}{\,\lrcorner\,}
\newcommand{\kk}{\Bbbk}%
\newcommand{\KK}{\mathbb{K}}%
\newcommand{\spann}[1]{\left\langle #1 \right\rangle}

\newcommand{\xx}{\mathbf{x}}%
\newcommand{\yy}{\mathbf{y}}%
\renewcommand{\aa}{\mathbf{a}}%
\newcommand{\bb}{\mathbf{b}}%
\newcommand{\Gmult}{\mathbb{G}_{m}}%
\newcommand{\Gadd}{\mathbb{G}_{a}}%
\newcommand{\Gx}{\mathbb{G}_{x}}%
\newcommand{\Gy}{\mathbb{G}_{y}}%
\newcommand{\Gdiag}{\mathbb{G}_{xy}}%
\newcommand{\Gres}{\mathbb{G}_{\mathrm{res}}}%
\newcommand{\Gmultbar}{\overline{\mathbb{G}}_{m}}%
\newcommand{\Hshort}{\mathcal{H}}%
\newcommand{\Hplus}{\Hshort^+}%
\newcommand{\DS}{S}%
\newcommand{\DT}{T}%
\newcommand{\DI}{I}%
\newcommand{\DR}{R}%
\newcommand{\DJJ}{J}%
\newcommand{\OO}{\mathcal{O}}%
\newcommand{\OOhat}{\hat{\mathcal{O}}}%
\newcommand{\mm}{\mathfrak{m}}%
\newcommand{\nn}{\mathfrak{n}}%
\newcommand{\pp}{\mathfrak{p}}%
\newcommand{\ppx}{\mathfrak{p}_x}%
\newcommand{\ppy}{\mathfrak{p}_y}%

\newcommand{\Hshortkk}{\Hshort_{\kk}}%
\newcommand{\HshortZZ}{\Hshort_{\mathbb{Z}}}%
\newcommand{\HshortQQ}{\Hshort_{\mathbb{Q}}}%
\newcommand{\HshortFp}{\Hshort_{\mathbb{F}_p}}%
\newcommand{\HshortK}{\Hshort_{\KK}}%
\newcommand{\HshortZZflag}{\Hshort{\mathrm{flag}}_{\mathbb{Z}}}%
\newcommand{\Hpluskk}{\Hshort^+_{\kk}}%
\newcommand{\HplusZZ}{\Hshort^+_{\mathbb{Z}}}%

\newcommand{\BBname}{Bia{\l{}}ynicki-Birula}%
\newcommand{\DSplus}{\DS_+}%
\newcommand{\singula}{\mathfrak{S}}%
\newcommand{\TNTframe}{TNT frame}

\newcommand{\pts}{\mathrm{pts}}%
\newcommand{\Hilbsixteen}{\OHilb_{\pts}(\mathbb{A}^{16}_{\mathbb{Z}})}%
\newcommand{\Hilbn}{\OHilb_{\pts}(\mathbb{A}^{n}_{\mathbb{Z}})}%
\newcommand{\HilbnCC}{\OHilb_{\pts}(\mathbb{A}^{n}_{\mathbb{C}})}%
\newcommand{\Hilbnkk}{\OHilb_{\pts}(\mathbb{A}^{n}_{\kk})}%
\newcommand{\fixedptemb}{i}%

\begin{document}

\title{Pathologies on the Hilbert scheme of points}
\author{Joachim Jelisiejew}
\date{\today{}}
\thanks{Institute of Mathematics, Polish Academy of Sciences and Institute of Mathematics, Faculty of
    Mathematics, Informatics and Mechanics, University of
    Warsaw. Email:
    \url{jjelisiejew@mimuw.edu.pl}.  Partially supported by NCN grant
2017/26/D/ST1/00913.}

\begin{abstract}
    We prove that the Hilbert scheme of points on a higher dimensional affine
    space is non-reduced and has components lying entirely in characteristic
    $p$ for all primes $p$.
    In fact, we show that Vakil's Murphy's Law holds up to retraction for this
    scheme. Our main tool is a generalized version of the \BBname{} decomposition.
\end{abstract}
\maketitle

\section{Introduction}
Vakil defined \emph{singularity type} as an equivalence
class of pointed schemes
under the relation generated by $(X, x) \sim (Y, y)$ if there is a
smooth morphism $(X, x)\to (Y, y)$. He defined that \emph{Murphy's Law
holds} for a moduli space if every singularity type of finite type over
$\mathbb{Z}$ appears on this space. In~\cite{Vakil_MurphyLaw} Vakil gave numerous
examples when this happens. A notable item missing in his list is the
Hilbert scheme of points. In fact little was known about its singularities:
for example the following classical questions raised by Fogarty and Hartshorne
were open.
\begin{question}[{\cite[p.~520]{fogarty}, \cite[Problem~1.25]{aimpl}, \cite{CEVV}}]\label{que:nonreducedness}
    Is $\Hilbn$ reduced for all $n$?
    Is $\HilbnCC$ reduced for all $n$?
\end{question}
\begin{question}[{\cite[p.~148]{HarDeform}}, {\cite[Problem~1.2]{aimpl}},
    \cite{Langer__zerodimensional}]\label{que:charp}
    Do all finite $\kk$-schemes, for a finite field $\kk$, lift to characteristic zero?
\end{question}
Question~\ref{que:nonreducedness} was completely open, and
$\OHilb_4(\mathbb{A}^3_{\mathbb{C}})$ is the only known reduced but singular Hilbert scheme of
points~\cite{roggero_albert, Bertone__double_generic}. It
was explicitly asked in~\cite[Problem~1.6]{aimpl} whether
$\OHilb_8(\mathbb{A}^4_{\mathbb{C}})$ is
reduced.
There was a bit of progress on Question~\ref{que:charp} in recent
years. Classically,
Berthelot and Ogus~\cite[\S3]{Ogus} note that the maximal ideal of the algebra
$\mathbb{F}_p[x_1, \ldots ,x_6]/(x_1^p, x_2^p, \ldots , x_6^p, x_1x_2 + x_3x_4
+x_5x_6)$ admits no divided power structure. Bhatt~\cite[3.16]{Bhatt1}
mentioned that
this algebra does not lift to $W_2(\mathbb{F}_p) = \mathbb{Z}/p^2$ and
Zdanowicz~\cite[\S3.2]{Zdanowicz_liftability_mod_p2} gave a short direct
proof of this fact.
Langer~\cite{Langer__zerodimensional} proved that for $p=2$ this algebra
\emph{does} lift to characteristic zero. He also refined Zdanowicz's method to
give, for every local Artin
ring $A$ with $pA\neq 0$ and residue characteristic $p$, a finite $\mathbb{F}_p$-scheme
nonliftable to $A$. The constructed scheme depends on $A$ and Langer writes
``in principle these schemes could be liftable to characteristic $0$ over some
more ramified rings but we are unable to check whether this really happens.''

In this paper we prove that the answers to
Questions~\ref{que:nonreducedness}-\ref{que:charp} above are negative; both pathologies occur as special
cases of a Murphy-type Law, which we now describe. We say that \emph{Murphy's Law holds up to retraction} for a space
$\mathcal{M}$ if for every singularity type $\singula$ there is a representative $(Y,
y)$ of $\singula$, an open subscheme $(X, x)$ of $\mathcal{M}$ and a retraction $(X, x) \to (Y, y)$.
Here, a \emph{retraction} $(X, x) \to (Y, y)$ is a morphism of pointed schemes
together with a section.
The aim of this paper is to prove the following theorem.
\begin{theorem}\label{ref:mainthm:thm}
    Murphy's Law holds up to retraction for $\Hilbsixteen$.
\end{theorem}
On the infinitesimal level, Theorem~\ref{ref:mainthm:thm} implies that for every singularity
type $\singula$, there exists a representative $(Y, y)$ of $\singula$ and a
point on $\Hilbsixteen$ with complete local ring $\OOhat_{Y, y}[[t_1, \ldots
,t_r]]/I$, for some $r$ and $I$ such that $\OOhat_{Y, y}\cap I = 0$. In particular, $\OOhat_{Y, y}$ is
a \emph{subring} of the complete local ring.

Murphy's Law up to
retraction holds also for the scheme
$\OHilb_{\pts}(\mathbb{P}^{16}_{\mathbb{Z}})$, as $\Hilbsixteen$ is its open
subscheme. More generally, the forgetful functor from embedded to
abstract deformations of a finite scheme is
smooth~\cite[p.~4]{artin_deform_of_sings}, hence the above pathologies appear
also for the abstract deformations of finite schemes and for Hilbert schemes
of points on every smooth quasi-projective variety of dimension at least
sixteen.

The negative answers to Questions~\ref{que:nonreducedness}-\ref{que:charp}
with $n=16$ follow from Theorem~\ref{ref:mainthm:thm} applied to
singularity types $[\Spec(\mathbb{Z}[u]/u^2), (u)]$ and $[\Spec(\mathbb{Z}/p),
0]$
respectively, see Section~\ref{sec:after}. More
precisely, for Question~\ref{que:charp} we obtain finite
schemes $R$ which do not lift to any ring $A$ with $pA \neq
0$. In Section~\ref{sec:after} we give explicit examples of this behavior,
for $\kk = \mathbb{Z}/p$, $p=2,3,5$. Using
the singularity type $[\Spec(\mathbb{Z}/p^{\nu}), (p)]_{{\nu}=2,3, \ldots }$ we also obtain
finite $(\mathbb{Z}/p)$-schemes $R$ that do lift to $W_{\nu}(\mathbb{Z}/p) =
\mathbb{Z}/p^{\nu}$
but do not lift to any $A$ with $p^{\nu}A \neq 0$ and so forth.

\smallskip
\newcommand{\varX}{X}%
\newcommand{\Xplus}{X^+}%
\newcommand{\embdim}{r}%
The main difficulty with analysing finite schemes is their lack of structure:
for example,
they admit no non-trivial line bundles. Our proof of
Theorem~\ref{ref:mainthm:thm} proceeds by a series of reductions from objects with
more structure, such as projective schemes. The main role is played by the generalized \BBname{}
decomposition, which we now recall.

Classically~\cite{BialynickiBirula__decomposition}, for a smooth variety $\varX$ with a $\Gmult$-action and
$\varX^{\Gmult} = \coprod_{i=1}^r Y_i$ with $Y_i$ connected, the
BB decomposition is a variety $\Xplus =
\coprod_{i=1}^r
\Xplus_i$ with a map $\theta_0\colon \Xplus \to \varX$ and a retraction
$\pi\colon \Xplus \to \varX^{\Gmult}$ which makes $\Xplus_i$ a locally trivial
affine fiber bundle over
$Y_i$.

The generalized \BBname{} decomposition~\cite{drinfeld,
Jelisiejew__Elementary, jelisiejew_sienkiewicz__BB} extends this construction to
$\Gmult$-schemes $\varX$, not necessarily smooth, normal or reduced. We apply it
to the scheme $\HshortZZ :=
\OHilb_{\pts}(\mathbb{A}^{\embdim}_{\mathbb{Z}})$ with the standard $\Gmult$-action
on $\mathbb{A}^{\embdim}_{\mathbb{Z}}$. We obtain a locally finite type
$\mathbb{Z}$-scheme $\HplusZZ = \OHilb^+_{\pts}(\mathbb{A}^{\embdim}_{\mathbb{Z}})$
together with a map $\theta_0\colon \HplusZZ \to
\HshortZZ$ and a retraction $\pi\colon \HplusZZ \to \HshortZZ^{\Gmult}$ with
section $i$ ($\HplusZZ$ is not in general a bundle over $\HshortZZ^{\Gmult}$).
The $\kk$-points in the image of $\theta_0$ are exactly the $\kk$-schemes
supported at the origin of $\mathbb{A}^{\embdim}_{\kk}$, so $\im(\theta_0)$ is nowhere
dense. To remedy this, we define $\theta\colon \Gadd^{\embdim} \times\HplusZZ \to \HshortZZ$
which maps $(v, [\DR])$ to the scheme $\theta_0([\DR])$ translated by $v$.
    The structure is summarized in the diagram below, see
    Section~\ref{sec:BBdecompositions} for details.
    \begin{equation}\label{eq:diagramHplus}
        \begin{tikzcd}
            \Gadd^{\embdim} \times \HplusZZ \arrow[d,
            "\pr_2"]\arrow[r, "\theta"]& \HshortZZ \\
            \HplusZZ \arrow[u, "0_{\Gadd^{\embdim}} \times \id", bend left]\arrow[ru, "\theta_0"']\arrow[d, "\pi"] & \\
            \HshortZZ^{\Gmult}\arrow[u, "\fixedptemb", bend left]
        \end{tikzcd}
    \end{equation}

    The maps $\theta_0$ and $\theta$ are injective on $\kk$-points for all
    fields $\kk$ and we identify the points of $\HplusZZ$ with their images in $\HshortZZ$.
    The tangent space to $[\DR]\in
    \Hshortkk(\kk)$ is equal to $\Hom_\DT(I_{\DR}, \DT/I_{\DR})$, where $\DT =
    H^0(\mathbb{A}^{\embdim}, \OO_{\mathbb{A}^{\embdim}_{\kk}})$. If $[\DR]$ is $\Gmult$-fixed, then
    this space is graded and in this case $T_{\Hpluskk, [\DR]} = \Hom_T(I_{\DR},
    \DT/I_{\DR})_{\geq 0}$. It follows that $d\theta_{[\DR]}$ is surjective if and only if
    \begin{equation}\label{ref:TNTcondition}
        \dim_{\kk} \Hom_{\DT}(I_{\DR}, \DT/I_{\DR})_{<0} = \dim \DT.
    \end{equation}
    If~\eqref{ref:TNTcondition} holds, we say that $\DR$ has \emph{trivial negative tangents}
    (\emph{TNT} for short). If $\DR$ has TNT, then $\theta$ is an open immersion on its
    neighbourhood, see~\cite{Jelisiejew__Elementary}. This property is
    important enough to give it a name: we say that \emph{$(X, x)$ locally
    retracts to $(Y, y)$} if
    there are open neighbourhoods $x\in U \subset X$ and $y\in V\subset Y$
    and a retraction $(U, x)\to (V, y)$. For example, if $\DR$ has TNT then
    $(\HshortZZ, [\DR])$ locally retracts to $(\HshortZZ^{\Gmult}, [\DR])$.

    Now we discuss the proof of Theorem~\ref{ref:mainthm:thm}.  We first
    present a natural but unsuccessful
    line of argument and then refine it to obtain a proof.
    Fix a singularity type $\singula$. Vakil~\cite{Vakil_MurphyLaw} proved
    that there is a smooth surface $Z \subset \mathbb{P}^4$ whose embedded
    deformations in $\mathbb{P}^4$ are of type $\singula$. For
    $M\gg 0$ the $\Gmult$-equivariant deformations of the
    cone $V(I(Z)_{\geq M}) \subset \mathbb{A}^5$ are also
    of singularity type $\singula$. Let $\nn$ be the ideal of the origin in
    $\mathbb{A}^5$.
    Then the
    $\Gmult$-equivariant deformations of the zero-dimensional truncation $R_0
    := V(I(Z)_{\geq M} + \nn^{M+2})$ are of singularity type $\singula$, see~\cite{erman_Murphys_law_for_punctual_Hilb}.
    In other words, $(\HshortZZ^{\Gmult}, [R_0])$ has type $\singula$.
    Thus, $(\HplusZZ, [R_0])$ has singularity type $\singula$ up to retraction.
    However, there is no reason for $\theta$ to be an open immersion around
    $[R_0]$ and we cannot say anything about the singularity type of $[R_0]\in
    \HshortZZ$.  To obtain a scheme such that that $\theta$ is an open
    immersion in its neighbourhood, we need a
    refinement that produces a scheme with TNT.

    The refinement is based on the concept of TNT frames.
    Let $\DI \subset \DS = \kk[x_1, \ldots ,x_n]$ be a homogeneous ideal and
    $a\geq 2$. A
    \emph{\TNTframe{} of size $a$ for $\DI$} is an ideal $\DJJ
    \subset \DT = \DS[y_1, \ldots , y_n] = \kk[x_1, \ldots ,x_n, y_1, \ldots ,y_n]$ given by
    \begin{equation}\label{eq:TNTframe}
        \DJJ := \DI\cdot \DT + (x_1, \ldots
        ,x_n)^{a+1} + (y_1, \ldots , y_n)^{2} + (x_1y_1 +  \ldots  +
        x_ny_n).
    \end{equation}
    Informally, the \TNTframe{} is a bifurcated reduction of $\DI$ to
    dimension zero. The quadric $Q = \sum x_iy_i$ is technically useful
    for the proof that $\Spec(\DT/\DJJ)$ has TNT because the deformations of $\DJJ$ inside $V(Q)$ do not admit
    any negative
    tangents under mild assumptions on $\DI$, see Corollary~\ref{ref:minusOneTangents:cor}.
    Since $\DJJ$ is homogeneous with respect to both $x_i$'s and $y_i$'s, the
    stabilizer of $[\DJJ]$ contains a two dimensional torus $\Gx \times \Gy$.
    Let $\Gdiag \subset \Gx \times \Gy$ be the torus acting diagonally.

The concept of \TNTframe{}s is geared towards the following result.
\begin{proposition}\label{ref:retractionFromFrame:prop}
    Let $\charr \kk\neq 2$ and let $\DI\subset \DS$ be a homogeneous ideal with $I_2 = 0$.
    Assume that $\depth(\DS_+, \DS/\DI) \geq 3$ and $\dim \DS \geq 3$. Then
    \begin{enumerate}
        \item\label{it:firstRetraction} $\Spec(\DT/\DJJ)$ has TNT, hence $(\HshortZZ, [\DJJ])$ locally retracts
            to $(\HshortZZ^{\Gdiag}, [\DJJ])$,
        \item\label{it:secondRetraction} $(\HshortZZ^{\Gdiag}, [\DJJ])$ locally retracts to
            $(\HshortZZ^{\Gx \times \Gy}, [\DJJ])$,
        \item\label{it:thirdRetraction} $(\HshortZZ^{\Gx \times \Gy}, [\DJJ])$
            locally retracts to $(\OHilb^{\Gx}(\mathbb{A}^n), [\DI + (x_1, \ldots
            ,x_n)^{a+1}])$.
    \end{enumerate}
\end{proposition}
For $a\gg 0$ the $\Gx$-equivariant deformations of $\DI + (x_1, \ldots
,x_n)^{a+1} \subset \DS$ and $\DI \subset \DS$ are canonically
isomorphic~\cite{erman_Murphys_law_for_punctual_Hilb}. For such a number $a$, the
composition of the local retractions from
Proposition~\ref{ref:retractionFromFrame:prop} gives a local retraction of $(\HshortZZ, [\DJJ])$ to
$(\OHilb^{\Gx}(\mathbb{A}^n), [\DI])$.
The analogue of Proposition~\ref{ref:retractionFromFrame:prop} holds also in
$\charr \kk = 2$, for a slightly modified $\DJJ$, see
Section~\ref{ssec:characteristicTwo}.

We return to the proof of Theorem~\ref{ref:mainthm:thm}. Fix a singularity type $\singula$,
a surface $Z \subset \mathbb{P}^4$ and its truncation $V(I(Z)_{\geq M})
\subset \mathbb{A}^5$ as above. The ideal $I(Z)_{\geq M}$ does not satisfy the
depth assumption of Proposition~\ref{ref:retractionFromFrame:prop}, so fix a linear embedding $\mathbb{A}^5 \into
\mathbb{A}^8$ and consider the extended ideal $\DI := I(Z)_{\geq M}\cdot \DS$ in $\DS =
H^0(\mathbb{A}^8, \OO_{\mathbb{A}^8})$. Let $\DJJ$ be a \TNTframe{} for $\DI$
with $a\gg0$. Proposition~\ref{ref:retractionFromFrame:prop} implies that
the (not necessarily equivariant!) deformations of $[\DJJ]$ in
$\Spec(\DT) = \mathbb{A}^{16}$ locally retract to the $\Gx$-equivariant
deformations of $[\DI]$ in $\mathbb{A}^8$. These in turn retract to the equivariant deformations of
$[I(Z)_{\geq M}]$ in $\mathbb{A}^5$. Thus,
$(\Hilbsixteen, [\DJJ])$ locally retracts to a scheme
of singularity type $\singula$ and the proof is concluded.

In the course of the proof we give two side-results. First, in
Corollary~\ref{ref:TNTforJ:prop} we present a class of ideals, generalizing
\TNTframe{}s above, which have TNT.
By~\cite[Theorem~1.2]{Jelisiejew__Elementary}, each of those ideals lies on an
elementary component of the Hilbert scheme; thus we obtain a new, very large class of
elementary components. Second, in Corollary~\ref{ref:rigidity:cor} we prove that
thickenings of maximal linear subspaces of the quadric $V(Q)$ are rigid.


Theorem~\ref{ref:mainthm:thm} together with related combinatorics, in particular
\TNTframe{}s which subtly balance prescribed
homogeneous deformations and the TNT condition, is the main novelty of this paper. The \BBname{}
decomposition of $\Hilbnkk$, in the equicharacteristic setting, was introduced
in~\cite{Jelisiejew__Elementary}.
The ambient dimension $n = 16$ in Theorem~\ref{ref:mainthm:thm} is chosen to make the argument transparent and
is probably not optimal. Hazarding a guess, we would say that
it can be reduced to $n= 6$ or even $n=4$. In
any case, the case $n = 3$ is special because of the superpotential
description~\cite{Szendroi_Dimca, Bryan_Behrend_Szendroi}; it would be
interesting to know the answers to
Questions~\ref{que:nonreducedness}-\ref{que:charp} in this case.

The above results exhibit pathologies of the space of based rank $n$
algebras, see~\cite[\S4]{poonen_moduli_space}.
Hilbert schemes of points also appear prominently in the study of secant and
cactus varieties and in algebraic complexity~\cite{Landsberg__book}. The
non-reducedness of $\HilbnCC$ strongly suggests that the equations for these
varieties obtained in~\cite{nisiabu_jabu_cactus} are only set-theoretic, not
ideal-theoretic. To make this suggestion rigorous, one would need to know
that the Gorenstein locus of $\HilbnCC$ is non-reduced, but this remains open.
Another interesting open question is whether there are generically non-reduced components
of the Hilbert scheme of points.

The outline is as follows: in Section~\ref{sec:linearalgebra} we prove the
necessary prerequisite results on the tangent spaces and maps. In
Subsections~\ref{ssec:negativeTangents}-\ref{ssec:degreeZeroTangents} we deal
with $\charr\kk\neq 2$, while
in Subsection~\ref{ssec:characteristicTwo} we present the modified
construction for characteristic two.  Section~\ref{sec:BBdecompositions}
contains main ideas of the paper: we discuss \BBname{} decompositions, prove a
generalized version of Proposition~\ref{ref:retractionFromFrame:prop} and
finally prove Theorem~\ref{ref:mainthm:thm}. In Section~\ref{sec:after} we
discuss consequences of specific singularity types and give explicit examples
of non-reduced points on the Hilbert scheme and components lying in positive
characteristic.

\subsection*{Acknowledgements}

I am very grateful to Piotr Achinger, Jaros{\l{}}aw Buczy{\'n}ski and Maciek Zdanowicz for helpful and inspiring
conversations and insightful comments of the early versions of this paper. The paper was
prepared during the Simons Semester \emph{Varieties: Arithmetic and
Transformations} which is supported by the grant 346300 for IMPAN from the
Simons Foundation and the matching 2015-2019 Polish MNiSW fund.

\smallskip
\section{Pointed schemes and smooth equivalence}
A \emph{pointed scheme} $(X, x)$ is a scheme $X$ of finite type over $\mathbb{Z}$
together with a point $x$ of the underlying topological space of $X$. A
\emph{morphism of pointed schemes} $f\colon (X, x)\to (Y, y)$ is a morphism of
schemes $f'\colon X\to Y$ such that $f'(x) = y$. We say that $f$ is
\emph{smooth} if the underlying morphism $f'\colon X\to Y$ is smooth.
We say that pointed schemes $(X, x)$ and $(Y, y)$ are \emph{smoothly
equivalent} if there exists a pointed scheme $(Z, z)$ and smooth maps
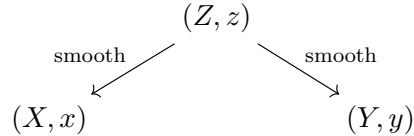
\[
    \begin{tikzcd}
        & (Z, z)\arrow[ld, "\mathrm{smooth}"']\arrow[rd, "\mathrm{smooth}"] &\\
        (X, x) & & (Y, y)
    \end{tikzcd}
\]
\begin{lemma}\label{ref:equivRel:lem}
    Being smoothly equivalent is an equivalence relation of pointed schemes.
\end{lemma}
\begin{proof}
    The relation is clearly reflexive and symmetric. To prove transitivity, suppose that the pairs $(X, x)$, $(Y, y)$ and $(Y, y)$, $(T,
            t)$ are smoothly equivalent. By definition there exist pointed schemes
            $(Z, z)$ and $(V, v)$ together with smooth maps forming the
            diagram
            
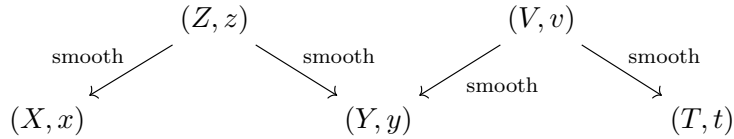
\[
                \begin{tikzcd}
                    & (Z, z)\arrow[ld, "\mathrm{smooth}"']\arrow[rd,
                    "\mathrm{smooth}"] & & (V, v)\arrow[ld, "\mathrm{smooth}"]\arrow[rd,
                    "\mathrm{smooth}"]\\
                    (X, x) & & (Y, y) & & (T, t)
                \end{tikzcd}
            \]
            Let $W := Z \times_Y V$ and let $\kappa(-)$ denote the residue
            field. The algebra $\kappa(z)\tensor_{\kappa(y)} \kappa(v)$ is a
            tensor product of nonzero algebras over a field, hence it is
            nonzero. Therefore the scheme $\Spec(\kappa(z)\tensor_{\kappa(y)}
            \kappa(v)) \simeq \Spec(\kappa(z))\times_{\Spec(\kappa(y))}
            \Spec(\kappa(v))$ is nonempty.
            Choose a point $p$ of this scheme.
            The natural maps
            $\Spec(\kappa(z))\to Z$ and $\Spec(\kappa(v))\to V$ induce a map
            $\Spec(\kappa(p))\to W$. Let $w\in W$ be the image of $p$. The
            pointed scheme $(W, w)$
            comes with smooth maps to $(X, x)$ and $(T, t)$ and proves that
            those schemes are smoothly equivalent.
\end{proof}
Lemma~\ref{ref:equivRel:lem} shows in particular that the above definition of
smooth equivalence
agrees with Vakil's one given in the
introduction.
The
equivalence classes of smooth equivalence are called \emph{singularity types}. The singularity type
of $(X, x)$ is denoted by $[(X, x)]$ or simply $[X]$ if $X$ has only one
point. For example, the singularity type $[\Spec(\mathbb{F}_p)]$ consists of pointed
schemes $(X, x)$ over $\Spec(\mathbb{F}_p)$ and such that $x\in X$ is smooth.
A \emph{retraction} $(X, x)\to (Y, y)$ is a pair $(f, s)$ where $f\colon (X,
x)\to (Y, y)$ and $s\colon (Y, y)\to (X, x)$ are morphisms of pointed schemes such that $f\circ s =
\id_{Y}$. For each such retraction, the residue fields of $x$ and $y$ are isomorphic.
The retractions we encounter in this article come from diagrams similar
to~\eqref{eq:diagramHplus}.

\section{Tangent spaces}\label{sec:linearalgebra}
\newcommand{\tang}[1]{T_{[#1]}\Hshort}%
\newcommand{\mmx}{\mm_{x}}%
\newcommand{\mmy}{\mm_{y}}%
\newcommand{\adx}{\alpha}%
\newcommand{\bdx}{\beta}%
\newcommand{\mmxpow}{\mmx^{a+1}}%
\newcommand{\mmypow}{\mmy^{b+1}}%

    In this section we prove tangent-map-surjectivity lemmas, such as the TNT
    condition, needed for
    the proof of a generalized version of Proposition~\ref{ref:retractionFromFrame:prop}.
    Specifically, the aim is to prove Proposition~\ref{ref:TNTforJ:prop},
    Corollary~\ref{ref:degreeZeroMain:cor}
    (for characteristic two, respectively Proposition~\ref{ref:TNTforJChar2:prop},
    Corollary~\ref{ref:degreeZeroMainCharTwo:cor}),
    which are applied in Section~\ref{sec:BBdecompositions}.
    These results follow from the chain of quite technical
    partial results, obtained using linear algebra and representation theory. We encourage the reader to consult
    Section~\ref{sec:BBdecompositions} for motivation before diving into
    details.

    Throughout, let $\kk$ be field. Let $\DI\subset \DS= \kk[\xx]$ be a
    homogeneous ideal. Let $\DT = \DS[\yy] = \kk[x_1, \ldots ,x_n, y_1, \ldots
    , y_n]$.
    Fix $a\geq 2$, $b\geq 1$ and let $\mmx := (x_1, \ldots ,x_n) \subset
    \DT$, $\mmy := (y_1, \ldots ,y_n) \subset \DT$, $Q := \sum x_iy_i$ and
    \[
        \DJJ := \DI\cdot \DT + \mmxpow + \mmypow + (Q).
    \]
    The ideal $\DJJ$ is $\mathbb{N}^2$-graded by $(\deg \xx, \deg \yy)$.
    Throughout this section the
    word \emph{homogeneous} for elements of $\DT$ refers to this bi-grading.
    For elements of $\DS$, the word \emph{homogeneous} refers to the usual grading by
    the total degree (when viewing $\DS$ as subring of $\DT$, these agree).
    The ideal $\DJJ$ is generated by elements of degrees $(*, 0)$, $(0, b+1)$,
    and $(1, 1)$. The space $\Hom_{\DT}(\DJJ, \DT/\DJJ)$ is graded by
    \begin{equation}\label{eq:tangentJ}
        \Hom_{\DT}(\DJJ, \DT/\DJJ)_{(\adx, \bdx)} := \left\{ \varphi\colon
            \DJJ \to \DT/\DJJ\ |\ \varphi\left(\DJJ_{(\gamma, \delta)}\right) \subset
            (\DT/\DJJ)_{(\gamma+\adx, \delta+\bdx)}\mbox{ for all }\gamma,
        \delta \right\}.
    \end{equation}

    This section is devoted to showing that certain homogeneous pieces of
    $\Hom_{\DT}(\DJJ, \DT/\DJJ)$ vanish or are as small as possible.
    The following Table~\ref{eq:tableCharNonTwo} provides for each such piece a reference to the
    corresponding result. Stars denote pieces which are
    not of interest. The table is subtly asymmetric with respect to $\xx$ and
    $\yy$.
    \begin{table}[h]
        \begin{tabular}{@{} cc | cc cc cc cc @{}}
                $\Hom_{\DT}(\DJJ, \DT/\DJJ)_{(\adx, \bdx)}$                &&
                $\beta \leq -2$ && $\beta = -1$ && $\beta = 0$ && $\beta\geq
                1$\\
                \toprule
                $\alpha\leq -2$ &&
                Cor~\ref{ref:verynegative:cor}  && Cor~\ref{ref:verynegative:cor} && Cor~\ref{ref:verynegative:cor} && Cor~\ref{ref:verynegative:cor}\\
            $\alpha = -1$&& Cor~\ref{ref:verynegative:cor} &&
            Lem~\ref{ref:minusminusTangents:lem} &&
            Cor~\ref{ref:minusOneTangents:cor} && $\star$\\
            $\alpha = 0$&& Cor~\ref{ref:verynegative:cor} &&
            Cor~\ref{ref:minusOneTangents:cor} && $\star$ && $\star$\\
            $\alpha = 1$&& Cor~\ref{ref:verynegative:cor} &&
            Cor~\ref{ref:homogeneousTangentsForJ:cor} && $\star$&& $\star$\\
            $\alpha \geq 2$&& Cor~\ref{ref:verynegative:cor} &&
            $\star$ && $\star$&& $\star$\\
        \end{tabular}
        \caption{Reference for computations of pieces of
            $\Hom_{\DT}(\DJJ, \DT/\DJJ)_{(\adx, \bdx)}$, characteristic not
            equal to two. The pieces marked with
            Cor~\ref{ref:verynegative:cor} or
            Lem~\ref{ref:minusminusTangents:lem} vanish, others do not.}\label{eq:tableCharNonTwo}
    \end{table}

    Now we introduce two tricks which recur in our computation of the
    homogeneous components of the space~\eqref{eq:tangentJ}.
    Consider the canonical epimorphism $p\colon \DT/(\DI\cdot \DT +
    (Q)) \to\DT/\DJJ$ and note that there exists a unique homogeneous linear
    section of the map $p$:
    \[
        s\colon \frac{\DT}{\DJJ} \to \frac{\DT}{\DI\cdot \DT + (Q)}
    \]
    and that $s$ is zero in degrees $(*, \geq a+1)$, $(\geq b+1, *)$ and an
    isomorphism in degrees $(\leq a, \leq b)$.
    Our first trick is the following lemma.
    \newcommand{\gensFstIndx}{c_{1i}}%
    \newcommand{\gensSndIndx}{c_{2i}}%
    \newcommand{\relsFstIndx}{d_{1j}}%
    \newcommand{\relsSndIndx}{d_{2j}}%
    \begin{lemma}[lifting homogeneous homomorphisms]\label{ref:lifting:lem}
        Let $N$ be a graded $\DT$-module with a minimal presentation
        \[
            \begin{tikzcd}
                \bigoplus_j \DT(-\relsFstIndx, -\relsSndIndx) \arrow[r,
                "\delta_1"] & \bigoplus_i
                \DT(-\gensFstIndx, -\gensSndIndx) \arrow[r, "\delta_0"] & N \arrow[r] & 0
            \end{tikzcd}
        \]
        and $\psi$ be a homomorphism $\psi\colon N\to \DT/\DJJ$ of degree $(\adx, \bdx)$.
        Suppose that $\relsFstIndx + \adx \leq a$ and $\relsSndIndx + \bdx \leq b$ for all $j$.
        Then $\psi$ lifts to a homomorphism $N\to \DT/(\DI\cdot \DT + (Q))$ of
        degree $(\adx, \bdx)$.
    \end{lemma}
    \begin{proof}
        Let $K := \frac{\DJJ}{\DI\DT+(Q)}$. Take a lifting of
        $\psi$ to a homogeneous chain complex map
        \[
            \begin{tikzcd}
                &\bigoplus \DT(-\relsFstIndx, -\relsSndIndx) \arrow[d,
                "\rho"]\arrow[r, "\delta_1"] & \bigoplus_i
                \DT(-\gensFstIndx, -\gensSndIndx) \arrow[r, "\delta_0"]\arrow[d] & N
                \arrow[r]\arrow[d, "\psi"] & 0\\
                0\arrow[r] & K \arrow[r] & \frac{\DT}{\DI\DT + (Q)} \arrow[r] &
                \frac{\DT}{\DJJ}\arrow[r] & 0
            \end{tikzcd}
        \]
        The map $\rho$ maps the generator of $\DT(-\relsFstIndx, -\relsSndIndx)$ into
        $K_{\relsFstIndx + \adx, \relsSndIndx + \bdx}$. But $\relsFstIndx +
        \adx\leq a$ and $\relsSndIndx + \bdx \leq b$, hence $K_{\relsFstIndx + \adx, \relsSndIndx + \bdx}
        = 0$. Therefore $\rho = 0$, which completes the claim.
    \end{proof}

    Our second trick is as follows: suppose that $K \subset \DT$ is an ideal
    and $M$ is a $\DT$-module with $\depth(K, M) \geq 2$. Then $\Ext^i(\DT/K,
    M) = 0$ for $i=0,1$, see~\cite[Proposition~18.4]{Eisenbud}, so $M =
    \Hom_{\DT}(\DT, M)\to \Hom_{\DT}(K, M)$ is an isomorphism. To make the
    trick applicable to $K = \DI\cdot \DT + (Q)$, we give a lower bound of the
    depth of ${\DT/(\DI\cdot \DT + (Q))}$, as follows.
    \begin{lemma}\label{ref:Qregular:lem}
        Let $d = \depth(\DS_+, \DS/\DI)$.
        Suppose that $d\geq 2$ and $f_1, \ldots
        ,f_{d-1}\in
        \DS_+$ is a regular sequence on $\DS/\DI$ consisting of homogeneous
        elements. Then $Q$ is a non-zero
        divisor on both $\DT/(\DI\cdot \DT)$ and $\DT/(\DI\cdot \DT + f_1\cdot
        \DT)$. Moreover, $f_1, \ldots
        ,f_{d-1}$ is a regular sequence on $\DT/(\DI\cdot \DT + (Q))$.
    \end{lemma}
    \begin{proof}
        By base change, we may assume that $\kk$ is algebraically closed.

        The quotient module $M = \DS/(\DI + (f_1, \ldots
        ,f_{d-1}))$ has
        $\DS_+$-depth at least one, hence there exists a quadric $f_d = \sum_{i=1}^n
        x_il_i\in \DS_2$ which is a regular element for $M$. Thus, the
        sequence
        \[
            (y_1 - l_1, y_2 - l_2, \ldots , y_n - l_n, f_1, \ldots ,f_{d-1}, Q)
        \]
        is regular for the $\DT$-module $\DT/(\DI\cdot \DT)$. This sequence
        consists of elements homogeneous with respect to
        the total degree, hence each of its permutations is also a regular
        sequence~\cite[Theorem~16.3]{Matsumura_CommRing}. In particular,
        the sequences $(f_1, Q, f_2, \ldots , f_{d-1}, y_1 -l_1, \ldots , y_n
        - l_n)$ and
        $(Q, f_1, \ldots , f_{d-1}, y_1 - l_1, \ldots , y_n - l_n)$ are regular for $\DT/(\DI\cdot
        \DT)$. The first one
        implies that $Q$ is regular on $\DT/(\DI\cdot \DT + (f_1))$ and the second that $Q$ is regular on
        $\DT/(\DI\cdot \DT)$ and that $f_1, \ldots , f_{d-1}$ is regular on $\DT/(\DI\cdot \DT +
        (Q))$.
    \end{proof}
    Now we begin direct computations of specific degrees of the tangent space
    at $\DJJ$.

    \subsection{Negative tangents}\label{ssec:negativeTangents}

    In this section we verify that $\Spec(\DT/\DJJ)$ has TNT.
    \begin{lemma}[$\DI$-ignoring lemma]\label{ref:Iignore:lem}
        Suppose that $\depth(\DSplus, \DS/\DI) \geq 2$.  Let
        $\adx$, $\bdx$ be integers, with $\adx \leq -1$. Let $\varphi\in
        \Hom(\DJJ, \DT/\DJJ)_{(\adx,
        \bdx)}$ be a homomorphism such that $\varphi(\mmxpow) = 0$. Then
        $\varphi(I) = 0$.
    \end{lemma}
    \begin{proof}
        Choose any homogeneous $i\in \DI$ and let
        $a'$ be its degree. If $a'\geq
        a+1$, then $i\in \mmxpow$ and so $\varphi(i) = 0$. Otherwise, we have
        $\mmx^{a+1-a'}i \subset \mmxpow$, so
        \[
            0 = \varphi\left( \mmx^{a+1-a'}i \right) = \mmx^{a+1-a'}\varphi(i).
        \]
        But $\deg(\DS_{a+1-a'}\varphi(i)) = (a+1+\adx, \bdx)$ and $a+1+\adx <
        a+1$. Therefore,
        \[
            \left(\frac{\DT}{\DI\cdot \DT + (Q)}\right)_{(a+1+\adx,\, \bdx)} \to
            \left(\frac{\DT}{\DJJ}\right)_{(a+1+\adx,\, \bdx)}
        \]
        is an isomorphism. Hence, $\varphi(i)$ uniquely lifts to an element
        $F$ of $\DT/(\DI\cdot \DT + (Q))$ such that $\mmx^{a+1-a'} F = 0$.
        By Lemma~\ref{ref:Qregular:lem}, we have
        $\depth(\mmx, \DT/(\DI\cdot \DT + (Q))) > 0$, hence $F = 0$, so
        $\varphi(i) = 0$. As $i\in I$ was chosen arbitrarily, we have $\varphi(I) = 0$.
    \end{proof}

    \begin{lemma}[depth lemma]\label{ref:depth:lem}
        Suppose that $\depth(\DSplus, \DS/\DI) \geq 3$ and $n\geq 3$. Fix
        $(\adx, \bdx)$ so that $\adx < 0$ or $\bdx < 0$.
        Let $\ppx$, $\ppy$ be homogeneous ideals with radicals $\mmx$, $\mmy$
        respectively.
        Then
        \[
            \Hom\left(\ppx, \frac{\DT}{I\cdot \DT + (Q)}\right)_{(\adx,
            \bdx)} = 0
            \quad\mbox{and}\quad
            \Hom\left(\ppy, \frac{\DT}{I\cdot \DT + (Q)}\right)_{(\adx,
            \bdx)} = 0
        \]
    \end{lemma}
    \begin{proof}
        Let $M =  \DT/(I\cdot \DT + (Q))$ viewed as a $\DT$-module.
        By Lemma~\ref{ref:Qregular:lem}, $\depth\left(\ppx, M\right) \geq 2$.
        Hence $\Ext^i(\DT/\ppx, M) = 0$ for $i = 0, 1$, so
        \[
            \Hom_{\DT}(\DT, M) \to \Hom_{\DT}(\ppx, M)
        \]
        is an isomorphism. But $\Hom_{\DT}(\DT, M) =M$ has no non-zero
        elements of degree $(\adx, \bdx)$ with $\adx < 0$ or $\bdx < 0$.  The
        same argument applies with $\ppy$ instead of $\ppx$; the depth assumption is
        satisfied as the sequence $(y_1, y_2)$ is regular.
    \end{proof}

    \begin{corollary}\label{ref:verynegative:cor}
        Suppose that $\depth(\DSplus, \DS/\DI) \geq 3$ and $n\geq 3$. Then
        $(\tang{\DJJ})_{(\adx, \bdx)} = 0$ for $\adx\leq -2$ or $\bdx\leq -2$.
    \end{corollary}
    \begin{proof}
        Consider the case $\adx \leq -2$. Take an element $\varphi\in\Hom(\DJJ,
        \DT/\DJJ)_{(\adx, \bdx)}$. By degree reasons, $\varphi$ sends $\mmypow$ and $Q$ to zero.
        Consider $\varphi' = \varphi_{|\mmxpow}\colon \mmxpow \to
        \DT/\DJJ$.
        This map sends generators of $\mmxpow$ to elements of $(\DT/J)_{(\leq
        a-1, \leq b-1)}$ and the syzygies of $\mmxpow$ are linear, of degree
        $(1, 0)$. Hence, by Lemma~\ref{ref:lifting:lem}, the map $\varphi'$ lifts to a map
        \[
            \varphi''\colon \mmxpow \to \frac{\DT}{I\cdot \DT + (Q)},
        \]
        of degree $(\adx, \bdx)$ with $\adx \leq -2$. Such a map is zero by
        Lemma~\ref{ref:depth:lem} applied to $\ppx = \mmxpow$. Therefore, $\varphi(\mmxpow) = 0$.
        From Lemma~\ref{ref:Iignore:lem} it follows that $\varphi(I) = 0$,
        hence $\varphi = 0$. The case $\beta\leq -2$ is symmetric, minus the
        use of Lemma~\ref{ref:Iignore:lem}.
    \end{proof}

    Now we will analyse homomorphisms of degree $(\adx, \bdx)$ with $\adx +
    \bdx = -1$. These do exist (e.g.~the tangents corresponding to
    the $\Gadd^{2n}$-action by translations), so the depth considerations as above
    are not directly applicable.
    \begin{lemma}\label{ref:quadricx:lemma}
        Suppose that $\depth(\DSplus, \DS/\DI) \geq 2$.
        Then
        \[
            \Hom_{\DT}\left( \frac{\mmxpow}{\mmxpow \cap (Q)}, \frac{\DT}{\DJJ}
            \right)_{(-1, 0)} = 0.
        \]
    \end{lemma}
    \begin{proof}
        Let
        \[
            \varphi\in \Hom_{\DT}\left( \frac{\mmxpow}{\mmxpow \cap (Q)}, \frac{\DT}{\DJJ}
            \right)_{(-1, 0)}
        \]
        Take the
        unique lift of $\varphi$ to a linear map $\varphi'\colon \DS_{a+1} \to
        (\DS/\DI)_a$.
        Let $m\in \DS_{a}$ be a monomial. Then $\deg(mQ) = (a+1, 1)$ and
        \[
            0 = \varphi(mQ) = \sum y_i\varphi(mx_i),
        \]
        hence there exists a form $n_m\in (\DS/I)_{a-1}$ such that
        \[
            \sum y_i\varphi'(mx_i) \equiv Qn_m = \sum y_i x_in_m \mod
            \DI\DT+\mmxpow,
        \]
        so $\varphi(mx_i) \equiv x_in_m\mod \DI\DT+\mmxpow$.
        We define a
        linear map $\psi\colon \DS_a \to (\DS/\DI)_{a-1}$ by $\psi(m) := n_m$.
        Suppose that $m_1, m_2\in \DS_{a}$ are monomials such that $x_im_1 =
        x_jm_2$ for some $i, j$. Then
        \begin{align*}
            x_i\psi(m_1) - x_j\psi(m_2) = x_in_{m_1} - x_jn_{m_2} &\equiv
            \varphi(x_im_1) - \varphi(x_jm_2) \\ &= \varphi(x_im_1 - x_jm_2) =
            \varphi(0) = 0\mod \DI\DT+\mmxpow.
        \end{align*}
        But $\deg(x_i\psi(m_1) - x_j\psi(m_2)) = a$, hence $x_i\psi(m_1) -
        x_j\psi(m_2)\in \DI\DT$
        and so $\psi$ extends to a $\DS$-module homomorphism $\psi'\colon
        \DS_{\geq a} \to \DS/\DI$ of degree $-1$. But $\depth(\DS_+, \DS/\DI)
        \geq 2$, so $\Ext^i(\DS/\DS_{\geq a}, \DS/\DI) = 0$ for $i=1,2$, hence
        $\Hom(\DS, \DS/\DI) \to \Hom(\DS_{\geq a}, \DS/\DI)$ is an
        isomorphism. In particular, there are no homomorphisms of negative
        degrees, so $\psi' = 0$. Accordingly, $\psi = 0$, hence $\varphi = 0$.
    \end{proof}

    \begin{lemma}\label{ref:quadricy:lemma}
        Let $n\geq 2$. Then
        \[
            \Hom_{\DT}\left( \frac{\mmypow}{\mmypow \cap (Q)}, \frac{\DT}{\DJJ}
            \right)_{(0, -1)} = 0.
        \]
    \end{lemma}
    \begin{proof}
        The proof of Lemma~\ref{ref:quadricx:lemma} can be repeated with $\xx$
        interchanged with $\yy$ and $I = 0$.
    \end{proof}

    Let $\partial_{x_1}, \ldots , \partial_{x_n}, \partial_{y_1}, \ldots
    ,\partial_{y_n}$ be the derivations with respect to variables of
    $\DT$. By Leibniz's rule, each
    such derivation $\partial$ induces a $\DT$-linear map $\DJJ\to \DT/\DJJ$.
    This map is an element of the tangent space $\tang{\DJJ}$. By slight abuse, we
    denote it by
    $\partial$. Geometrically, these
    elements arise from the action of $\Gadd^{2n}$ on $\mathbb{A}^{2n} =
    \Spec(\DT)$ by translation.
    \begin{corollary}\label{ref:minusOneTangents:cor}
        Suppose that $\depth(\DSplus, \DS/\DI) \geq 2$ and $n\geq 2$. Then
        \[
            (\tang{\DJJ})_{(-1, 0)} = \spann{\partial_{x_1}, \ldots ,
            \partial_{x_n}}\quad\mbox{and}\quad(\tang{\DJJ})_{(0, -1)} =
            \spann{\partial_{y_1}, \ldots , \partial_{y_n}}
        \]
    \end{corollary}
    \begin{proof}
        Choose any homomorphism $\varphi\colon \DJJ\to \DT/\DJJ$ of degree
        $(-1, 0)$. Then $\varphi(Q) \in \spann{y_1, \ldots ,y_n}$. But
        $\partial_{x_i}(Q) = y_i$, hence there is a unique linear combination
        $D$ of $\left\{ \partial_{x_i} \right\}$ such that $(\varphi-D)(Q) =
        0$. Replacing $\varphi$ by $\varphi-D$, we may assume $\varphi(Q) =
        0$. By Lemma~\ref{ref:quadricx:lemma}, we have $\varphi(\mmxpow) = 0$.
        By Lemma~\ref{ref:Iignore:lem}, we have $\varphi(I) = 0$. Finally,
        $\varphi(\mmypow) = 0$ by degree reasons, so $\varphi = 0$ and
        the claim follows for $(\tang{\DJJ})_{(-1, 0)}$. The argument for
        degree $(0, -1)$ is symmetric.
    \end{proof}

    Out of all tangents of negative degrees, there is a single degree left to
    consider: $(-1, -1)$.
    \begin{lemma}[$(-1, -1)$ tangents]\label{ref:minusminusTangents:lem}
        Suppose that $\mmx^{a} \not\subset \DI$.
        Then $(\tang{\DJJ})_{(-1, -1)} = 0$.
    \end{lemma}
    \begin{proof}
        Let $\varphi\colon \DJJ\to \DT/\DJJ$ a homomorphism of degree $(-1,
        -1)$. Then $\varphi(\mmxpow) = \varphi(\mmypow) = \varphi(I) = 0$ by
        degree reasons. Moreover, $\varphi(Q)\in \kk$. If $\varphi(Q) \neq 0$,
        then
        \[
            \mm^a = \mm^{a}\varphi(Q) = \varphi(\mm^{a}Q) \subset \varphi(\mmxpow) =
            0,
        \]
        so $\mm^a\subset \DI$, a contradiction. Hence $\varphi(Q) = 0$, so
        $\varphi=0$.
    \end{proof}

    \begin{proposition}[TNT for $\DJJ$]\label{ref:TNTforJ:prop}
        Suppose that $\depth(\DSplus, \DS/\DI) \geq 3$ and $n\geq 3$.
        Then the scheme $\Spec(\DT/\DJJ)$ has TNT.
    \end{proposition}
    \begin{proof}
        This follows from Corollary~\ref{ref:verynegative:cor},
        Corollary~\ref{ref:minusOneTangents:cor},
        and Lemma~\ref{ref:minusminusTangents:lem}.
    \end{proof}

    \subsection{Degree zero tangents in characteristic $\neq
    2$}\label{ssec:degreeZeroTangents}
    \newcommand{\DV}{V}%
    \newcommand{\DVdual}{V^{*}}%
    \newcommand{\GLV}{\GL(V)}%
    \newcommand{\SLV}{\SL(V)}%
    \newcommand{\Endlin}{\mathbb{M}_{n\times n}}%
    \newcommand{\Phibar}{\overline{\Phi}}%

    Proposition~\ref{ref:TNTforJ:prop} is the essential part to obtain the
    local retraction from Point~\ref{it:firstRetraction} in Proposition~\ref{ref:retractionFromFrame:prop}. To
    obtain the retraction from Point~\ref{it:secondRetraction}, we need to
    compute the tangents of $[\Spec(\DT/\DJJ)]$ that have degree $(\adx, -\adx)$ for
    $\adx > 0$.
    All homomorphisms coming from such tangents kill
    $\DI$ by degree reasons. Hence, the material in this section is
    independent of the choice of $\DI$.
    To emphasise this further, we
    introduce the linear space $\DV = \spann{x_1, \ldots ,x_n}$ and identify
    $y_1, \ldots ,y_n$ with the dual space $\DVdual$. Then
    \[
        \DT_{(\adx, \bdx)} \simeq \Sym^\adx \DV \tensor \Sym^{\bdx} \DVdual.
    \]
    for all $\adx, \bdx$ and, in particular, $Q$ becomes the trace element in
    $\DV \tensor \DVdual$, hence is $\SLV$-invariant.
    For $\aa, \bb\in \mathbb{N}^n$ we let $\xx^{\aa}$, $\yy^{\bb}$ denote the
    monomials $x_1^{\aa_1} \ldots x_n^{\aa_n}$, $y_1^{\bb_1} \ldots
    y_n^{\bb_n}$ respectively. We will denote the space $(\Sym^{\adx}
    \DVdual)^*$ by $\DP^{\adx} \DV$ and denote by $\xx^{[\aa]}\in
    \DP^{\adx} \DV$ the functional which is dual to $\yy^{\aa}$.
    There is an contraction action
    $(-)\hook(-)\colon\DVdual
    \tensor \DP^{\adx} \DV \to \DP^{\adx-1} \DV$ given by
    \begin{equation}\label{eq:definitionOfContraction}
    (\ell \hook \varphi)(f) = \varphi(\ell\cdot f)
    \end{equation}
    for all $\ell\in \DVdual$,
    $\varphi\in \DP^{\adx}\DV$ and $f\in \Sym^{\adx-1} \DVdual$.
    For $\charr \kk > \adx$ there is a unique $\GLV$-equivariant isomorphism
    $\DP^{\adx} \DV  \simeq \Sym^{\adx} \DV$ and it sends $\xx^{[\aa]}$ to
    $\xx^{\aa}/\aa!$. Under this isomorphism, contraction corresponds to differentiation
    $\ell \circ \varphi \mapsto \partial_{\ell}(\varphi)$.
    See~\cite[A2.4]{Eisenbud} or~\cite[\S3]{nisiabu_jabu_cactus} for details.

    Now we introduce the group $G$ giving trivial degree zero tangents; this is
    the degree-zero counterpart of $\Gadd^{2n}$.
    Let $G \subset \GL_{2n, \mathbb{Z}}$ be a subgroup given in the basis $x_1, \ldots ,x_n,
    y_1,  \ldots , y_n$ by
    \begin{equation}\label{eq:Gdef}
        G = \left\{ \begin{pmatrix}I_n & A \\ 0 & I_n\end{pmatrix}\ |\ A\in
            \Endlin\right\}.
        \end{equation}
    \newcommand{\mathg}{\mathfrak{g}}%
    \newcommand{\mathgprim}{\mathfrak{g}'}%
    The group $G$ is smooth and acts naturally on $\HshortZZ$ and
    $\HshortZZ^{\Gdiag}$. The Lie algebra $\mathg$ of $G_{\kk}$ maps $y_i$'s to
    combinations of $x_j$'s. Hence, the tangent to the orbit map $G_{\kk}\ni g\mapsto g\cdot
    [\DJJ]\in \Hshortkk^{\Gmult}$ is
    \begin{equation}\label{eq:tangentToG}
        \mathg \to \Hom_{\DT}(\DJJ, \DT/\DJJ)_{(1, -1)}.
    \end{equation}
    Let $G'$ be the stabilizer of $Q$ in $G$. In the
    description~\eqref{eq:Gdef}, it consists of anti-symmetric
    matrices $A$. Let $\mathgprim$ be the Lie algebra of $G'$.
    The action of $\mathgprim$ annihilates $Q$ and we obtain a tangent map
    \begin{equation}\label{eq:tangentToGprim}
        \mathgprim\to \Hom_{\DT}(\DJJ/Q, \DT/\DJJ)_{(1, -1)}.
    \end{equation}

    Now we proceed to computations. Throughout this subsection, $\tensor$ denotes
    $\tensor_{\kk}$.
    The $\DS$-module $\DT_{*, b+1}$ is free. Let $M := (\DT/Q)_{*, b+1}$.
    This is an $\DS$-module with presentation
    \begin{equation}\label{eq:presentationOfM}
        0 \to \DS(-1, -b-1)\tensor \Sym^{b} \DVdual \to \DS(0, -b-1) \tensor
        \Sym^{b+1} \DVdual
        \to M,
    \end{equation}
    where the twists correspond to the fact that generators of $M$ have
    $\xx$-degree $(0, b+1)$ and its syzygies have degree $(1, b+1)$ with
    respect to natural bi-grading.
    The presentation map is just the multiplication by $Q$. Explicitly, it is
    given by
    \begin{equation}\label{eq:definitionOfMapping}
        f\tensor g \mapsto \sum_{i=1}^{n} x_if\tensor y_ig.
    \end{equation}
    The module $M$ plays a key role in the computation of $(1, -1)$ tangents of
    $[\Spec(\DT/\DJJ)]$. There is a natural injection $M \subset \mmypow$, hence we obtain
    restriction maps
    \begin{equation}\label{eq:degreeZeroRestriction}
        \begin{tikzcd}
            \Hom_{\DT}\left( \frac{\DJJ}{Q}, \frac{\DT}{\mmypow + (Q)} \right)_{(1,
            -1)} \arrow[r]\arrow[d] & \Hom_{\DT}\left( \frac{\mmypow}{\mmypow \cap
            (Q)}, \frac{\DT}{\mmypow + (Q)} \right)_{(1, -1)} \arrow[r]\arrow[d] & \Hom_{\DS}\left(M,
        \frac{\DT}{Q}\right)_{(1, -1)}\arrow[d, "r"]\\
            \Hom_{\DT}\left( \frac{\DJJ}{Q}, \frac{\DT}{\DJJ} \right)_{(1,
            -1)} \arrow[r] & \Hom_{\DT}\left( \frac{\mmypow}{\mmypow \cap
            (Q)}, \frac{\DT}{\DJJ} \right)_{(1, -1)} \arrow[r, "r'"] & \Hom_{\DS}\left(M,
        \frac{\DT}{\DJJ}\right)_{(1, -1)}.
        \end{tikzcd}
    \end{equation}
    \begin{lemma}\label{ref:degreezeroReductionToS:lem}
        Suppose $I_2 = 0$. Then all maps in the diagram~\eqref{eq:degreeZeroRestriction} are bijective.
    \end{lemma}
    \begin{proof}
        Left-side horizontal maps are bijective, because the homomorphisms kill other
        generators of $\DJJ/Q$ by degree reasons. The right-side horizontal maps
        are bijective since homomorphisms of degree $(1, -1)$ into
        $\DT/\mmypow$ annihilate $\DT_{(*, \geq b+2)}$. Finally, $I_2 =
        0$ implies that the surjection $\DT/\DJJ\to \DT/Q$ is bijective in degrees
        $(\leq 1, \leq 2)$. Then the rightmost
        downward arrow is bijective by applying Lemma~\ref{ref:lifting:lem}
        to~\eqref{eq:presentationOfM}.
    \end{proof}
    We concentrate on analysing $\Hom_{\DS}(M, \DT/Q)_{(1,
    -1)}$. For brevity, let $K := \Hom_{\DS}\left(M, \frac{\DT}{Q}\right)_{(1, -1)}$. From the
    presentation~\eqref{eq:presentationOfM} we obtain an exact sequence
    \[
        0 \to K \to
        \Hom_{\DS}\left(\DS(0, -b-1)\tensor
        \Sym^{b+1} \DVdual, \frac{\DT}{Q}\right)_{(1, -1)} \to  \Hom_S\left(
        \DS(-1, -b-1) \tensor \Sym^{b+1} \DVdual,
        \frac{\DT}{Q}\right)_{(1, -1)},
    \]
    which simplifies to
    $0 \to K \to
    \DP^{b+1} \DV \tensor \left( \DT/Q \right)_{(1, b)} \to
    \DP^{b} \DV \tensor \left( \DT/Q \right)_{(2, b)}$.
    Hence, we obtain a commutative diagram with exact columns and bottom row.

    \begin{equation}\label{eq:diagramEquivariant}
        \begin{tikzcd}
            && 0\arrow[d] & 0\arrow[d]\\
            && \DP^{b+1}\DV \tensor
            \Sym^{b-1} \DVdual\arrow[d, "\id\tensor (-\cdot Q)"]
            \arrow[r, "\Phi_0"] & \DP^{b} \DV \tensor {\DV\tensor
                \Sym^{b-1} \DVdual}\arrow[d, "\id\tensor (-\cdot Q)"]\\
                && \DP^{b+1}\DV \tensor
                {\DV\tensor \Sym^{b} \DVdual}\arrow[d]
                \arrow[r, "\Phi"] & \DP^{b} \DV \tensor {\Sym^2 \DV\tensor
                    \Sym^{b} \DVdual}\arrow[d]\\
                    0 \arrow[r]& K \arrow[r]& \DP^{b+1}\DV \tensor
                    \frac{\DV\tensor \Sym^{b} \DVdual}{Q\cdot \Sym^{b-1}
                \DVdual}\arrow[d]
                \arrow[r, "\Phibar"] & \DP^{b} \DV \tensor \frac{\Sym^2 \DV\tensor
                    \Sym^{b} \DVdual}{Q\cdot V\tensor \Sym^{b-1}
                \DVdual}\arrow[d]\\
                && 0 & 0
            \end{tikzcd}
        \end{equation}
        Let us write down $\Phi$ and $\Phi_0$ explicitly.
        The map $\Phi$ comes from applying $\Hom_{\DS}(-, \DT)_{(1, -1)}$ to
        the map~\eqref{eq:definitionOfMapping}. Therefore, it is given in
        coordinates by
        \begin{equation}\label{eq:PhiConcrete}
            \Phi(\varphi\tensor f\tensor g) = \sum_{i=1}^n (y_i\hook
            \varphi)\tensor (x_if) \tensor g,
        \end{equation}
        where $\hook$ denotes
        contraction, as defined in~\eqref{eq:definitionOfContraction}.

    \begin{proposition}\label{ref:degreeZeroMain:prop}
        Suppose that $b = 1$ and $\charr \kk \neq 2$. Then the map
        \begin{equation}\label{eq:tangentToGprimDowngraded}
            \mathgprim \to \Hom_{\DS}(M, \DT/Q)_{(1, -1)}
        \end{equation}
        obtained by composing~\eqref{eq:tangentToGprim} and~\eqref{eq:degreeZeroRestriction} is
        bijective.
    \end{proposition}
    \begin{proof}
        For an element of $\mathgprim$, the image of $y_i$ is read off the image of
        $y_i^{b+1} = y_i^2$ in the corresponding homomorphism, so the
        map~\eqref{eq:tangentToGprimDowngraded} is injective and
        it is enough to check that $\dim \Hom_{\DS}(M, \DT/Q)_{(1, -1)}
        = \dim \mathgprim= \dim \Lambda^2 \DV$. We do this directly.

        Since $b = 1$, the map $\Phi\colon \DP^{2}\DV \tensor \DV\tensor
        \DVdual\to \DV \tensor \Sym^2 \DV\tensor \DVdual$ has source and
        target of the same dimension. We will prove that it is bijective.
        It is enough to prove surjectivity. By~\eqref{eq:PhiConcrete}, we have
        $\Phi = \Psi \tensor
        \id_{\DVdual}$ for the map $\Psi\colon \DP^{2} \DV \tensor
        \DV \to \DV \tensor S^2V$ given by
        \[
            \Psi(\varphi \tensor f) = \sum_{i=1}^n (y_i\hook \varphi)\tensor
            (x_if).
        \]
        It is enough to prove that $\Psi$ is surjective. For pairwise
        distinct $i$, $j$, $k$, we have
        \[
            \Psi\left((x_i\cdot x_j)\tensor x_k + (x_i\cdot x_k)\tensor x_j
            - (x_j\cdot x_k) \tensor x_i\right) =
            2x_i\tensor x_j\cdot x_k.
        \]
        The same holds for not necessarily distinct
        $i$, $j$, $k$ under the convention that $x_i\cdot x_i = 2x_i^{[2]}$ in
        $\DP^2 \DV$. Thus, the map $\Psi$ is surjective, hence the map $\Phi$ is bijective.
        In particular, $\Phi$ and $\Phi_0$ are injective. From the snake
        lemma applied to~\eqref{eq:diagramEquivariant}, we have $K  \simeq \coker\Phi_0$, in particular $\dim K = \dim
        \Lambda^2 \DV$ as claimed.
    \end{proof}
    \begin{remark}
        Proposition~\ref{ref:degreeZeroMain:prop} fails for $\charr \kk = 2$;
        in particular $\Psi$ and $\Phi$ are not injective.
    \end{remark}
    The restriction to $b = 1$, while sufficient for our purposes, is not very
    satisfactory. For $ b > 1$ we have the following result in large enough
    characteristics. We will not use it in the proof of
    Theorem~\ref{ref:mainthm:thm}, so we only sketch a proof.
    \begin{proposition}\label{ref:degreeZeroForCharZero:prop}
        Let $b\geq 1$ be arbitrary and $\charr \kk = 0$. Then the map
        $\mathgprim \to \Hom_{\DS}(M, \DT/Q)_{(1, -1)}$ is bijective.
    \end{proposition}
    \begin{proof}[Sketch of proof]
        By Proposition~\ref{ref:degreeZeroMain:prop} we may assume $b > 1$
        (we only need this for notational reasons in Schur functors).
        The maps in Diagram~\eqref{eq:diagramEquivariant} split into maps
        between simple $\GLV$-modules, which are indexed by Young
        diagrams~\cite[\S6]{fultonharris}.

        The cokernel of $\Phi$ is
        isomorphic to $\Sch_{b, 2} \DV\tensor \Sym^{b} \DVdual$. Similarly, the
        cokernel of $\id\tensor (-\cdot Q)$ is isomorphic to $\DP^{b+1} \DV
        \tensor \Sch_{b+1, b,  \ldots , b} \DV$. Hence, the cokernel of
        $\Phibar$ is obtained from $\DP^{b} \DV \tensor \DV^{\tensor 2}
        \tensor \Sym^{b} \DVdual  \simeq  \Sym^b \DV\tensor \DV^{\tensor 2} \tensor
        \Sym^b \DVdual$ by applying two partial symmetrizations
        (corresponding to dividing by images of $\id\tensor(-\cdot Q)$ and $\Phi$),
        which together imply that $\coker \Phibar = \Sch_{2b, b+2, b, \ldots ,
        b}\DV$. Now, a dimension count on the bottom row shows that $\dim K =
        \dim \Lambda^2 \DV$ as claimed.
    \end{proof}

    \begin{corollary}\label{ref:homogeneousTangentsForJ:cor}
        Suppose that $\DI_2 = 0$ and $\charr \kk \neq 2$.
        Suppose further that either $b= 1$ or $\charr\kk = 0$.
        Then the map~\eqref{eq:tangentToG} is
        bijective.
    \end{corollary}
    \begin{proof}
        \newcommand{\mathl}{\mathfrak{lt}}%
        \newcommand{\maths}{\mathgprim}%
        Let
        \[
            \mathl = \left\{\begin{pmatrix} 0 & A\\ 0 &
                0\end{pmatrix} \ |\ A\mbox{ is lower-triangular}
        \right\}\quad\mbox{and}\quad
            \maths = \left\{\begin{pmatrix} 0 & A\\ 0 &
                0\end{pmatrix} \ |\ A\mbox{ is anti-symmetric} \right\}
        \]
        Take an element $\varphi\in \Hom(\DJJ, \DT/\DJJ)_{(1, -1)}$.
        By degree reasons, $\varphi(\mmxpow) = \varphi(I) = 0$ and
        $\varphi(Q)\in \DT_{2, 0} = \DS_2$. There is a unique element
        $D\in \mathl$ such that $D(Q) = \varphi(Q)$. Replacing
        $\varphi$ by $\varphi - D$, we can assume $\varphi(Q) = 0$. Therefore,
        $\varphi$ comes from an element of
        \[
            \Hom_{\DT}\left(\frac{\mmypow}{\mmypow\cap (Q)},
            \frac{\DT}{\DJJ}\right)_{(1, -1)}.
        \]
        By Lemma~\ref{ref:degreezeroReductionToS:lem} and
        Proposition~\ref{ref:degreeZeroMain:prop} or
        Proposition~\ref{ref:degreeZeroForCharZero:prop}, there exists a unique
        element of $\maths$ mapping to $\varphi$, which concludes the proof.
    \end{proof}
    \begin{corollary}\label{ref:degreeZeroMain:cor}
        Let $n\geq 3$ and $\charr \kk \neq 2$.
        Suppose that $\depth(\DSplus, \DS/\DI) \geq 3$ and $\DI_2 = 0$.
        Suppose further that either $b= 1$ or $\charr\kk = 0$.
        The map
        \[
            \mathg \to \bigoplus_{\adx \geq 1} \Hom_{\DT}(\DJJ, \DT/\DJJ)_{(\adx, -\adx)}
        \]
        is bijective.
    \end{corollary}
    \begin{proof}
        This follows from Corollary~\ref{ref:homogeneousTangentsForJ:cor} and
        Corollary~\ref{ref:verynegative:cor}.
    \end{proof}

    The above computations of degree-zero tangents can be translated to
    a geometric statement, which is of some independent interest at least as a
    motivation.
    \newcommand{\Ly}{L_{\mathrm{y}}}%
    \begin{corollary}\label{ref:rigidity:cor}
        Let $\charr \kk \neq2$ and $n\geq 3$. Assume either $b =1$ or $\charr \kk = 0$.
        Let
        \[
            \Ly^{b} = V((y_1, \ldots ,y_n)^{b+1}) \subset V(Q) \subset
            \mathbb{P}^{2n-1}
        \]
        be a thickening of a projective subspace on a
        quadric. Then $\Ly^b$
        is a rigid scheme: we have $T^1_{\Ly^b} = 0$, where $T^1$ is the
        Schlessinger's functor~\cite[\S3]{HarDeform}.
    \end{corollary}
    \begin{proof}
        Let $B = \DT/((Q) + (y_1, \ldots ,y_n)^{b+1})$ be the coordinate ring
        of $\Ly^b$. Since $n\geq 3$, we have $\depth(\mmy, B) \geq 2$, so
        $T^1_{\Ly^b} \simeq (T^1_{B})_0$ naturally, see
        e.g.~\cite[Theorem~5.4]{HarDeform}.
        Corollary~\ref{ref:homogeneousTangentsForJ:cor} proves that
        $(T^1_B)_0 =0$, directly from the construction of $T^1_B$.
    \end{proof}

    \subsection{Tangents in characteristic $2$}\label{ssec:characteristicTwo}

    \newcommand{\DJJchtwo}{\mathcal{J}}%
    \begin{assumption}\label{ref:assumptionCharTwo}
        In this subsection $\kk$ is a field of characteristic two.
    \end{assumption}
    In this case, Proposition~\ref{ref:degreeZeroMain:prop} fails and
    we need to replace $\mmy^2$ in the definition of $\DJJ$ by another ideal.
    There are many possible replacements; in any case the symmetry has to be
    broken. We choose $\pp \subset \kk[\yy]$ defined by
    \[
        \pp := (y_1^2, y_2^2, \ldots , y_n^2) + (y_1)\cdot (y_1, y_2, y_3,  \ldots ,
        y_n) = (y_1^2, y_2^2, \ldots , y_n^2) + (y_1)\cdot\mmy,
    \]
    mainly for the (relatively) simple computations.
    Let $\DI \subset \DS = \kk[x_1, \ldots ,x_n]$ be a homogeneous ideal and
    $a\geq 2$. A \emph{tweaked frame} of size $a$ for $\DI$ is an ideal
    $\DJJchtwo
    \subset \DT = \kk[\xx,\yy]$ defined by
    \begin{equation}\label{eq:JdefChar2}
        \DJJchtwo := \DI\cdot \DT + \mmxpow + \pp + (Q).
    \end{equation}
    As in Subsection~\ref{ssec:negativeTangents}, we calculate some graded
    pieces of $\tang{\DJJchtwo} = \Hom_{\DT}(\DJJchtwo, \DT/\DJJchtwo)$. Most of the arguments will
    directly pass to this setup. There is some additional work needed in
    degrees $(*, -1)$. Anyway, for clarity we provide statements and sketches
    of proofs of all steps.

    An important additional piece is the following easy result about the
    syzygies of $(y_1^2, y_2^2,  \ldots , y_n^2, Q)$.
    \begin{lemma}\label{ref:syzygiesOfQ}
        Let $\alpha \in \{0, 1\}$ and $n\geq 3$. Suppose that $F_1, \ldots ,F_n\in \DT_{\adx, 1}$ are
        forms satisfying $\sum_{i=1}^n x_i^2 F_i \in (Q)$. Then $F_i \in (Q)$ for $i=1, \ldots ,n$.
        In particular, if $\alpha = 0$, then $F_i = 0$ for all $i$.
    \end{lemma}
    \begin{proof}
        Let $\sum x_i^2F_i = Q\cdot G$, then $\deg(G) = (\adx+1, 0)$.
        Differentiate with respect to $y_j$ to obtain
        \begin{equation}\label{eq:tmpquadricsyzygies}
            \sum x_i^2 \frac{\partial F_i}{\partial y_j} = x_j G.
        \end{equation}
        In follows that $G\in (x_1^2, \ldots , x_n^2) + (x_j)$. Intersecting over
        all $x_j$, we get $G = \sum_{i=1}^n \lambda_i x_i^2$ for some
        $\lambda_i\in \kk$.
        The forms $\frac{\partial
        F_i}{\partial y_j}$ have degree $(\adx, 0)\leq (1, 0)$, hence
        $\frac{\partial F_i}{\partial y_j} = \lambda_i x_j$.
        Consequently, we have $F_i = \sum_{j=1}^n \frac{\partial F_i}{\partial y_j} y_j =
        \lambda_i \sum_j x_jy_j \in (Q)$.
    \end{proof}

    \begin{lemma}\label{ref:verynegativeCharTwo:lem}
        Suppose that $\depth(\DSplus, \DS/\DI) \geq 3$ and $n\geq 3$. Then
        $(\tang{\DJJchtwo})_{(\adx, \bdx)} = 0$ for $\adx\leq -2$ or $\bdx\leq -2$.
    \end{lemma}
    \begin{proof}
        The non-Koszul syzygies of $\pp$ are linear, hence the proof of
        Corollary~\ref{ref:verynegative:cor} applies without changes.
    \end{proof}

    \begin{lemma}\label{ref:quadricChar2:lemma}
        Suppose that $n\geq 3$ and $I_2 = 0$. Then
        \begin{equation}\label{eq:degreeMinusOneChar2}
            \Hom_{\DT}\left( \frac{\pp}{\pp \cap (Q)}, \frac{\DT}{\DJJchtwo}
            \right)_{(0, -1)} = 0.
        \end{equation}
    \end{lemma}
    \begin{proof}
        Take a homomorphism $\varphi$ as in the left-hand side of~\eqref{eq:degreeMinusOneChar2} and
        its unique lift of $\varphi$ to a linear map $\varphi'\colon \DT_{0,
        2} \to \DT_{0, 1}$.
        First, from $Q^2= \sum_{i=1}^n x_i^2y_i^2$ we get
        \[
            \sum_{i=1}^n x_i^2 \varphi'(y_i^2) \in (Q),
        \]
        so $\varphi'(y_i^2)=0$ for $i=1, \ldots , n$, by Lemma~\ref{ref:syzygiesOfQ}.
        Next, $y_1Q = \sum_{i=1}^n x_i(y_1y_i)$ so we have
        $0 = \sum_{i=1}^n x_i\varphi(y_1y_i)$ and
        \[
            \sum_{i=1}^n x_i \varphi'(y_1y_i) = \lambda\cdot Q,
        \]
        for some $\lambda\in \kk$. Comparing coefficients near $x$'s, we
        get $\varphi'(y_1y_i) = \lambda y_i$.
        In particular $0 = \varphi'(y_1^2) = \lambda y_1$, hence $\lambda = 0$, thus
        $\varphi' = 0$.
    \end{proof}
    \begin{corollary}\label{ref:minusOneTangentsCharTwo:cor}
        Suppose that $\depth(\DSplus, \DS/\DI) \geq 2$, $I_2 = 0$ and $n\geq 3$. Then
        \[
            (\tang{\DJJchtwo})_{(-1, 0)} = \spann{\partial_{x_1}, \ldots ,
            \partial_{x_n}}\quad\mbox{and}\quad(\tang{\DJJchtwo})_{(0, -1)} =
            \spann{\partial_{y_1}, \ldots , \partial_{y_n}}
        \]
    \end{corollary}
    \begin{proof}
        Homomorphisms of degree $(-1, 0)$ kill $\pp$ and this case reduces to
        the one considered in Corollary~\ref{ref:minusOneTangents:cor}.
        Consider a homomorphism $\varphi\in \Hom_{\DT}(\DJJchtwo,
        \DT/\DJJchtwo)_{(0, -1)}$. Then $\varphi(Q) \in \spann{x_1, \ldots ,x_n}$. But
        $\partial_{y_i}(Q) = x_i$, hence there is a unique linear combination
        $D$ of $\left\{ \partial_{y_i} \right\}$ such that $(\varphi-D)(Q) =
        0$. Replacing $\varphi$ by $\varphi-D$, we may assume $\varphi(Q) =
        0$. By Lemma~\ref{ref:quadricChar2:lemma} we have $\varphi(\pp) = 0$.
        Moreover, $\varphi(I + \mmxpow) = 0$ for degree reasons. This shows
        that
        $\varphi = 0$.
    \end{proof}

    The following Proposition~\ref{ref:TNTforJChar2:prop} summarizes the above
    discussion. We stress once more, that we assume $\charr \kk = 2$, see
    Assumption~\ref{ref:assumptionCharTwo}.
    \begin{proposition}\label{ref:TNTforJChar2:prop}
        Suppose that $\depth(\DSplus, \DS/\DI) \geq 3$ and $n\geq 3$.
        Then the scheme $\Spec(\DT/\DJJchtwo)$ has TNT.
    \end{proposition}
    \begin{proof}
        This follows from Lemma~\ref{ref:verynegativeCharTwo:lem},
        Corollary~\ref{ref:minusOneTangentsCharTwo:cor},
        and a direct analogue of Lemma~\ref{ref:minusminusTangents:lem}.
    \end{proof}
    Now we analyse degree zero tangents of $\DJJchtwo$. In~\eqref{eq:Gdef} we
    defined the group $G \subset \GL_{2n, \mathbb{Z}}$ by
    \[
        G = \left\{ \begin{pmatrix}I_n & A \\ 0 & I_n\end{pmatrix}\ |\ A\in \Endlin\right\}.
    \]
    We also introduce its Lie algebra $\mathg$ and the tangent map
    \begin{equation}\label{eq:tangentToGCharTwo}
        \mathg \to \Hom_{\DT}(\DJJchtwo, \DT/\DJJchtwo)_{(1, -1)}.
    \end{equation}
    As before, we will prove that it is bijective.
    \begin{lemma}[special liftings]\label{ref:charTwoLiftingLemma:lem}
        Let $n\geq 4$.
        Let $\varphi\colon \DJJchtwo\to \DT/\DJJchtwo$ be a homomorphism of
        degree $(1, -1)$. Then $\varphi(y_i^2) = 0$ for all $i$.
        Moreover, there exists a special lifting of $\varphi_{|\pp}$, that is, a degree $(1, -1)$ linear map
        \[
            \varphi'\colon \spann{y_i^2, y_1y_i\ |\ i=1, \ldots , n} \to
        \DT_{1, 1}
        \]
        such that
        \begin{enumerate}
            \item $\varphi'(g)\mod (Q) = \varphi(g)$ for
                every generator of $\pp$,
            \item $\varphi'(y_i^2) = 0$ and $\varphi'(y_1y_i)\in (y_1, y_i)$
                for all $i$.
        \end{enumerate}
    \end{lemma}
    \begin{proof}
        Take a homomorphism $\varphi\colon \DJJchtwo/Q\to \DT/\DJJchtwo$ of
        degree $(1, -1)$ and any lift to a linear map $\varphi'\colon
        \DT_{0, 2} \to \DT_{1, 1}$. The syzygy $Q\cdot Q = \sum x_i^2y_i^2$ implies
        $\sum_{i=1}^n x_i^2 \varphi'(y_i^2)\in (Q)$. By
        Lemma~\ref{ref:syzygiesOfQ} we have $\varphi'(y_i^2) = 0$ for all $i$,
        after possible changing the lifting.

        The syzygy $y_i\cdot (y_1y_i) = y_1\cdot y_i^2$ implies that
        $y_i\varphi'(y_1y_i) \in \pp + (Q)$. Take $p\in \pp$ and $l\in
        \DT_{0, 1}$ such that
        \begin{equation}\label{eq:tmpimprovinglift}
            y_i \varphi'(y_1y_i) = p + lQ.
        \end{equation}
        Let $l = \sum_{j=1}^n l_j y_j$ for $l_j\in \kk$. Replacing $p$ with
        $p + l_1y_1Q$, we may assume $l_1 = 0$. Pick $j\neq i$ and any $j'\neq
        1, i, j$ (we use $n\geq 4$). The monomial
        $x_{j'}y_{j'}y_j$ appears in $lQ$ with coefficient $l_j$, but does
        not appear elsewhere in~\eqref{eq:tmpimprovinglift}, so $l_j =0$.
        Then $l = l_iy_i$. Then $p\in (y_i)\cap \pp = y_i(y_1, y_i)$, so $p =
        y_ip'$ for some $p'\in (y_1, y_i)$. Clearly, $\varphi'(y_1y_i) - p'\in
        (Q)\subset \DJJchtwo$, hence we may replace $\varphi'(y_1y_i)$ by
        $p'$. The lifting thus obtained satisfies the conditions.
    \end{proof}
    \begin{proposition}\label{ref:degreeZeroMainCharTwo:prop}
        Let $n\geq 4$ and $I_2 = 0$.
        The map~\eqref{eq:tangentToGCharTwo} is
        bijective.
    \end{proposition}
    \begin{proof}
        \def\Qhat{\hat{Q}}%
        Take a lift $\varphi'$ as in Lemma~\ref{ref:charTwoLiftingLemma:lem}
        and the unique lift $\Qhat\in \DT_{2, 0}$ of $\varphi(Q)$. Recall from
        Lemma~\ref{ref:charTwoLiftingLemma:lem} that $\varphi'(y_i^2) = 0$ and
        $\varphi(y_i^2)= 0$ for all $i$.
        Write
        \begin{equation}\label{eq:tmpgoodlift}
            \varphi'(y_1y_i) = y_i L_i + y_1 M_i\quad \mbox{for} \quad i=1, \ldots , n,
        \end{equation}
        where $L_1 = M_1 = 0$.
        Since
        $y_1Q = \sum_{i=1}^n x_i(y_1y_i)$, we have $y_1\varphi(Q) =
        \sum_{i=1}^n x_i\varphi(y_1y_i)$. Since $I_2 = 0$, we have
        \begin{equation}\label{eq:tmpsyzygy}
            y_1 \Qhat - \sum_{i=1}^n x_i \varphi'(y_1y_i) \in \mu\cdot Q.
        \end{equation}
        for some $\mu\in \DT_{1, 0}$. Putting~\eqref{eq:tmpgoodlift} into~\eqref{eq:tmpsyzygy}, we
        get
        \begin{equation}\label{eq:tmpsyzygytwo}
            y_1 \Qhat - \sum_{i=1}^n x_i(y_iL_i + y_1 M_i) = \mu Q.
        \end{equation}
        Comparing coefficients of $y_i$ in~\eqref{eq:tmpsyzygytwo}, for $i >
        1$, we obtain $L_2 = L_3 =  \ldots = L_n = -\mu$. Moreover,
        $\varphi'(y_1\cdot y_1) = \varphi'(y_1^2) = 0$ by the first paragraph,
        so $L_1 = 0$ and $M_1 = 0$. From the equation~\eqref{eq:tmpsyzygytwo},
        we compute $\Qhat = \sum_{i\geq 2} x_i M_i + \mu x_1$.

        Define $D\in \mathg$ by setting $D(y_1) = \mu$ and $D(y_i) = M_i$ for
        $i\geq 2$. Then $D(y_i^2) = 0 = \varphi'(y_i^2)$ and moreover
        $\varphi'(y_1y_i) = \mu y_i + y_1 M_i = D(y_1y_i)$ for all $i\geq 2$.
        Therefore $\varphi$ is the image of $D\in \mathg$.
    \end{proof}

    \begin{corollary}\label{ref:degreeZeroMainCharTwo:cor}
        Let $n\geq 4$ and $I_2 = 0$. The map
        \[
            \mathg \to \bigoplus_{\adx \geq 1}\Hom_{\DT}(\DJJchtwo,
            \DT/\DJJchtwo)_{(\adx, -\adx)}.
        \]
        is
        bijective.
    \end{corollary}
    \begin{proof}
        This follows from Proposition~\ref{ref:degreeZeroMainCharTwo:prop} and
        Lemma~\ref{ref:verynegativeCharTwo:lem}.
    \end{proof}

\section{\BBname{} decompositions and retractions}\label{sec:BBdecompositions}
    \newcommand{\HshortDR}{(\HshortZZ, [\DJJ])}%
    \newcommand{\Hflaghom}{(\HshortZZflag^{\Gx \times \Gy}, [\DI\DT+\mmxpow \subset \DJJ])}%
    \newcommand{\HshortDVhom}{(\HshortZZ^{\Gx}, [\DI+\mmxpow])}%
    \newcommand{\HshortDRhom}{(\HshortZZ^{\Gx \times \Gy}, [\DJJ])}%
    \newcommand{\HshortDRsemihom}{(\HshortZZ^{\Gdiag}, [\DJJ])}%
    In this section we formally define \BBname{} decompositions and apply
    the results from the previous section to obtain the local
    retractions and prove Proposition~\ref{ref:retractionFromFrame:prop} and Theorem~\ref{ref:mainthm:thm}.
    In total, these proofs apply three BB decompositions; we
    consider three different linear $\Gmult$-actions, which correspond to
    the three retractions from Proposition~\ref{ref:retractionFromFrame:prop}.

    Let $\Gmult :=
    \Spec(\mathbb{Z}[t^{\pm 1}])$ and $\Gmultbar :=
    \Spec(\mathbb{Z}[t^{-1}])  \simeq \mathbb{A}^1_{\mathbb{Z}}$ be its
    compactification at infinity. Every $\Gmult$-action on
    $\mathbb{A}^{\embdim}_{\mathbb{Z}}$ induces a $\Gmult$-action on $\HshortZZ$.
    The \emph{\BBname{} decomposition} of the $\Gmult$-scheme $\HshortZZ$ is
    a functor from $\mathbb{Z}$-schemes to sets given by
    \[
        \HplusZZ(B) := \left\{ \varphi\colon\Gmultbar\times B  \to \HshortZZ\ |\ \varphi
            \mbox{ is } \Gmult\mbox{-equivariant}\right\}.
    \]
    This functor is represented by a scheme $\HplusZZ$, whose connected
    components are quasi-projective over
    $\mathbb{Z}$. This scheme comes with naturals maps:
    \begin{itemize}
        \item \emph{Forgetting about the limit} by restricting $\varphi$ to
            $\varphi_{|1 \times B}\colon B\to \HshortZZ$. This induces a map
            $\theta_0\colon \HplusZZ\to \HshortZZ$.
        \item \emph{Restricting to the limit} by restricting $\varphi$ to
            $\varphi_{|\infty\times B}\colon
            B\to \HshortZZ$. The family $\varphi_{|\infty \times B}$ is equivariant,
            hence the image lies in $\HshortZZ^{\Gmult}$. We obtain a map
            \[
                \pi\colon \HplusZZ\to \HshortZZ^{\Gmult}.
            \]
        \item \emph{Embedding of fixed points}. The trivial $\Gmult$-action on
            $\HshortZZ^{\Gmult}$ extends to $\Gmultbar \times
            \HshortZZ^{\Gmult} \to \HshortZZ^{\Gmult}$ and hence induces a map
            $\fixedptemb\colon \HshortZZ^{\Gmult}\to \HplusZZ$. We have $\pi
            \circ \fixedptemb = \id$ and $\theta_0 \circ \fixedptemb\colon
            \HshortZZ^{\Gmult} \to \HshortZZ$ is the embedding of fixed points. In particular, $\pi$ is a
            retraction.
    \end{itemize}
    The existence of \BBname{} decompositions for Hilbert schemes of points
    and totally divergent $\Gmult$-action is
    proven in~\cite[Proposition~3.1]{Jelisiejew__Elementary};
    in that paper $\kk$ denotes a field, but the proof holds equally well for
    $\kk = \mathbb{Z}$: indeed the proof
    of~\cite[Proposition~3.1]{Jelisiejew__Elementary} goes through without changes
    for $\kk = \mathbb{Z}$ and its main nontrivial ingredient is the existence of
    the multigraded Hilbert scheme which holds over any commutative
    ring $\kk$~\cite{Haiman_Sturmfels__multigraded}. Alternatively and more
    explicitly, one can take the standard affine
    $\Gmult$-stable covering $\{U_{\lambda}\}_{\lambda}$ of $\HplusZZ$,
    see~\cite[\S18.1]{Miller_Sturmfels} and then glue $\HplusZZ$ from
    $U_{\lambda}^{+}$, as in~\cite[Proposition~5.3]{jelisiejew_sienkiewicz__BB};
    neither of these two steps depends on the base ring.
    \begin{remark}
        The image of $\theta_0$ is frequently nowhere dense. This happens in
        particular for the dilation action $\Gmult \times \mathbb{A}^{\embdim} \to
        \mathbb{A}^{\embdim}$, given by $t\cdot (x_1, \ldots ,x_{\embdim}) =
        (tx_1, \ldots , tx_{\embdim})$,
        see~\cite[Proposition~3.2]{Jelisiejew__Elementary}. Hence, we cannot
        in general hope to prove that $\theta_0$ is an open immersion.
        To remedy this, we choose a smooth algebraic group $\mathbb{Z}$-scheme $G$
        acting on $\mathbb{A}^{\embdim}$ and extend the map
        $\theta_0$ to
        \[
            \theta\colon G \times \HplusZZ \to \HshortZZ,
        \]
        which maps $(g, x)$ to $g\cdot \theta_0(x)$. Below, $G$ will be either
        $\Gadd^{\embdim}$ acting by translation or the unipotent group $G$ defined
        in~\eqref{eq:Gdef} and recalled below.
    \end{remark}

    Recall the group $G$ of linear
    transformations given in the basis $x_1, \ldots ,x_n,
    y_1,  \ldots , y_n$ by
    \[
        G = \left\{ \begin{pmatrix}I_n & A \\ 0 & I_n\end{pmatrix}\ |\ A\in
            \Endlin\right\}.
    \]
    and its Lie algebra $\mathg$.

    We would now like to apply \BBname{} decompositions to prove
    Proposition~\ref{ref:retractionFromFrame:prop} for frames (for $\charr \kk
    \neq 2$) and tweaked frames (for $\charr \kk = 2$). To avoid
    dichotomy in proofs and for clarity, we abstract the necessary properties into a
    standalone definition.
    \begin{definition}[Frame-like ideals]\label{ref:framelike:def}
        Let $\DI\subset \DS$ be a homogeneous ideal and $\DJJ \subset \DT =
        \kk[\xx, \yy]$ be a $\mathbb{N}^2$-homogeneous ideal of the
        form
        \[
            \DJJ = \DI\DT + \mmxpow + \pp + (Q),
        \]
        where $\pp \subset (y_1, \ldots ,y_n)^2$.
        We say that $\DJJ$ is \emph{frame-like} if the following conditions hold
        \begin{enumerate}[label=(\alph*)]
            \item\label{it:framelikeTNT} the scheme $\Spec(\DT/\DJJ)$ has TNT,
            \item\label{it:framelikeDegreeZero} the map $\mathg\to \bigoplus_{\adx \geq 1} \Hom_{\DT}(\DJJ,
                \DT/\DJJ)_{(\adx, -\adx)}$ is bijective,
            \item\label{it:framelikeSaturation} there exists a $b$ such that
                $\pp \supset \DT_{2, b+1}$ and $I_{b} = 0$.
        \end{enumerate}
    \end{definition}
    \begin{lemma}\label{ref:abstraction:lem}
        Let $\charr \kk\neq 2$ and $n = \dim \DS \geq 3$. Let $\DI\subset \DS$ be an ideal with
        $\depth(\DS_+, \DS/\DI) \geq 3$, $I_2 = 0$. Let $\DJJ$ be a frame of
        size $a$ for $\DI$. Then $\DJJ$ is frame-like (for all $b\geq 1$).
    \end{lemma}
    \begin{proof}
        This follows from Proposition~\ref{ref:TNTforJ:prop} and
        Corollary~\ref{ref:degreeZeroMain:cor}.
    \end{proof}
    \begin{lemma}\label{ref:abstractionCharTwo:lem}
        Let $\charr \kk = 2$ and $n = \dim \DS \geq 4$. Let $\DI\subset \DS$ be an ideal with
        $\depth(\DS_+, \DS/\DI) \geq 3$, $I_n = 0$. Let $\DJJchtwo$ be a tweaked frame of
        size $a$ for $\DI$. Then $\DJJchtwo$ is frame-like (for all $b\geq
        n$).
    \end{lemma}
    \begin{proof}
        This follows from Proposition~\ref{ref:TNTforJChar2:prop},
        Corollary~\ref{ref:degreeZeroMainCharTwo:cor} and from $(y_1^2, \ldots
        ,y_n^2) \supset \mmy^{n+1}$.
    \end{proof}

    Let us fix a frame-like ideal $\DJJ$.
    The ideal $\DJJ$ is $\mathbb{N}^2$-graded by $(\deg \xx, \deg \yy)$ and hence its
    stabilizer contains a two-dimensional torus $\Gmult^{2}$. We consider three
    of its one-dimensional sub-tori: $\Gx, \Gy, \Gdiag$. They act on $\DT$ by
    respectively
    \[
        t\cdot (\xx, \yy) = \begin{cases}
            (t\xx, \yy) & \mbox{for }t\in \Gx\\
            (\xx, t\yy) & \mbox{for }t\in \Gy\\
            (t\xx, t\yy) & \mbox{for }t\in \Gdiag
        \end{cases}
    \]
    We identify $\kk$-points of $\HshortZZ$ with
    finite subschemes of affine space and with their ideals, with the
    conversion that $\DI \subset \DS$ and $\DJJ \subset \DT$, so for example
    $(\HshortZZ, [\DJJ])$ is the pointed scheme
    $(\OHilb(\mathbb{A}^{2n}_{\mathbb{Z}}), [\DJJ])$.
    Finally, we consider the following Diagram~\eqref{eq:maindiag} of Hilbert schemes, where subscripts indicate
    the points of interest (see below for explanations).
    \begin{equation}\label{eq:maindiag}
        \begin{tikzcd}
            && \Gadd^{2n} \times (\HplusZZ, [\DJJ]) \arrow[r,
            "\theta_{xy}"]\arrow[d, "\pi_{xy}"] & \HshortDR\\
            &G \times \HshortDRsemihom^{+, \Gy}\arrow[r,
            "\theta_{x}"]\arrow[d, "\pi_x"] & \HshortDRsemihom\arrow[u,
            "\fixedptemb_{xy}", bend left] &\\
            \Hflaghom\arrow[r, "\pr_2"]\arrow[d,
            "\pr_1"'] & \HshortDRhom\arrow[u, "\fixedptemb_{x}", bend left] & &\\
            \HshortDVhom &
        \end{tikzcd}
    \end{equation}

    The scheme $\HshortDRsemihom^{+,
    \Gy}$ is the \BBname{} decomposition of $\HshortDRsemihom$ with respect to
    the natural $\Gy$-action. The scheme $\HshortZZflag$ is
    the flag Hilbert scheme, which parameterizes deformations of
    pairs of finite subschemes $R_1 \subset R_2$ in affine space (the flag
    Hilbert scheme is constructed as a closed subscheme of $\HshortZZ \times \HshortZZ$).

The map $\theta_{xy}$ was denoted by $\theta$ in the introduction. It is the forgetful
map $\Gadd^{2n} \times \HplusZZ$ to $\Gadd^{2n} \times \HshortZZ$ composed
with
the translation action $\Gadd^{2n} \times \HshortZZ \to \HshortZZ$. Similarly,
the map
$\theta_x$ is the forgetful map followed by the $G$-action on
$\HshortZZ^{\Gdiag}$.

The proof of Proposition~\ref{ref:retractionFromFrame:prop} is a journey on
Diagram~\eqref{eq:maindiag}, from its upper-right corner to the lower-left
one. Specifically, each of the three parts of
the proposition asserts the existence of a local retraction for one ``hook''
on this diagram. First two
retractions will be obtained from the BB decompositions corresponding to
$(\theta_{xy}, \pi_{xy}, i_{xy})$ and $(\theta_x, \pi_x, i_x)$ respectively.
The last one is easily deduced from $(\pr_1, \pr_2)$. The
conditions~\ref{it:framelikeTNT},~\ref{it:framelikeDegreeZero} of the definition of
frame-like ideal imply that $d\theta_{xy}$ and $d\theta_x$ are
bijective in relevant points.
\begin{proposition}\label{ref:quadriclemma:prop}
    There exists an open $\Gdiag$-stable neighbourhood $[\DJJ]\in U_{xy}
    \subset \HplusZZ$ such that $(\theta_{xy})_{|\Gadd^{2n} \times
    U_{xy}}\colon \Gadd^{2n} \times U_{xy} \to
    \HshortZZ$ is an open immersion.
\end{proposition}
\begin{proof}
    The map $\theta_{xy}$ at $(0,[\DJJ])$ induces an injection on obstruction
    spaces~\cite[Thm~4.2]{Jelisiejew__Elementary} and a bijection on tangent
    spaces by Condition~\ref{ref:framelike:def}\ref{it:framelikeTNT}.
    Hence it is \'etale at this point.
    The torus $\Gdiag$ normalizes $\Gadd^{2n}$ in the automorphism scheme of
    $\HshortZZ$, hence we obtain a semidirect subgroup $H := \Gdiag \ltimes
    \Gadd^{2n}\subset \Aut(\HshortZZ)$. This subgroup acts on
    $\Gadd^{2n} \times \HplusZZ$ by $(t, v_1) \cdot (v_2, x) :=
    (t(v_1+v_2)t^{-1}, tx)$ and $\theta_{xy}$ is $H$-equivariant.
    Hence, the
    \'etale locus of $\theta_{xy}$ has the form $\Gadd^{2n} \times U$ for some $\Gmult$-stable open $U \subset
    \HshortDR^{+, \Gdiag}$. Now, $\theta_{xy}$ is universally injective, hence $(\theta_{xy})_{|U}$ is an open immersion,
    see~\cite[Tag~025G]{stacks_project}. Take $U_{xy} := U$.
\end{proof}

To repeat the argument above for $\theta_x$ we need
to check universal injectivity for the $G$-action.
\begin{lemma}\label{ref:Gstabs:lem}
    Let $[\DJJ]\in (\HshortZZ^{\Gdiag})^{+, \Gy}$ be a point with limit
    point $[\DJJ_0]\in \HshortZZ^{\Gx \times \Gy}$. Assume that $\mathg \to \Hom_{\DT}(\DJJ_0, \DT/\DJJ_0)$
    is injective. Then the stabilizer of $[\DJJ]$ in $G$ is trivial.
\end{lemma}
\begin{proof}
    Suppose that there exists a point of the stabilizer of $[\DJJ]$ and
    identify it with a matrix $A$ as in~\eqref{eq:Gdef}. For every $t\in
    \Gmultbar$ the element of $G$ corresponding to the matrix $t^{-1}A$
    stabilizes $t\cdot [\DJJ]$. Therefore, the tangent vector $A\in \mathg$
    maps to zero in $\Hom_{\DT}(\DJJ_0, \DT/\DJJ_0)$. Since the map is injective, $A =
    0$.
\end{proof}

\begin{proposition}\label{ref:planesOnQuadric:prop}
    There exists an open $\Gy$-stable neighbourhood $[\DJJ] \in U_y \subset
    \HshortZZ^{\Gdiag}$ such that $(\theta_{x})_{|G \times U_y}\colon G \times U_y
    \to \HshortZZ^{\Gdiag}$ is an open immersion.
\end{proposition}
\begin{proof}
    Let $\mathg$ be the Lie algebra of $G$.
    The tangent map $d(\theta_x)$ at $[\DJJ]$ is
    \[
        \mathg \oplus \bigoplus_{\adx\geq 0}\Hom(\DJJ, \DT/\DJJ)_{(-\adx,
        \adx)}\to \bigoplus_{\adx\in \mathbb{Z}}\Hom(\DJJ, \DT/\DJJ)_{(-\adx,
        \adx)},
    \]
    so it is bijective by
    Condition~\ref{ref:framelike:def}\ref{it:framelikeDegreeZero}.
    As in
    the proof of Proposition~\ref{ref:quadriclemma:prop}, we find a
    $\Gy$-stable $U_y$ such that $(\theta_x)_{|G \times U_y}$ is \'etale. Then
    by Lemma~\ref{ref:Gstabs:lem} the map $(\theta_x)_{|G \times U_y}$ is universally
    injective, hence an open immersion~\cite[Tag~025G]{stacks_project}.
\end{proof}

\begin{proposition}\label{ref:JToPair:prop}
    The map $\pr_2$ is an isomorphism on the connected component of
    $[\DI\DT+\mmxpow \subset \DJJ]\in \HshortZZflag^{\Gx \times \Gy}$.
\end{proposition}
\begin{proof}
    \def\Ical{\widetilde{\DI}}%
    \def\Jcal{\widetilde{\DJJ}}%
    \def\gcal{\widetilde{g}}%
    \def\scal{\widetilde{s}}%
    This is an easy consequence of the bi-homogeneity of $\DJJ$.
    The map $(\DT/(\DI\DT + \mmxpow))_{(*, 0)} \to (\DT/\DJJ)_{(*, 0)}$ is bijective.
    Consider a base ring $A$ and an equivariant deformation $\DT_A/\Jcal$ of
    $\DT/\DJJ$ over $A$. Let $g_1, \ldots ,g_r$ be homogeneous generators of
    $\DI$ and $\gcal_1, \ldots , \gcal_r$ be their unique lifts in $\Jcal$.
    Let $\Ical := (\gcal_1, \ldots , \gcal_r)+\mmxpow$, then $(\DT_A/\Ical)_{(*,
    0)} \to (\DT_A/\Jcal)_{(*, 0)}$ is bijective as well.
    Take any syzygy $s\in \DT^r$ between $g_1, \ldots , g_r$. Since $\DT_A/\Jcal$ is
    $A$-flat, there is a syzygy $\scal\in \DT_A^r$ lifting $s$. This means that
    $\scal(\gcal_1, \ldots , \gcal_r) \in \Jcal$. We have $\deg(s) =
    \deg(\scal) = (*, 0)$, so $\scal(\gcal_1, \ldots ,
    \gcal_r) \in \Ical$. Hence,
    $[\Ical \subset \Jcal]$ is the required deformation
    of $[\DI\DT+\mmxpow \subset \DJJ]$. Uniqueness is evident.
\end{proof}

\begin{proposition}\label{ref:IToPair:prop}
    The projection $\pr_1$ is a retraction of
    the connected component of $[\DI\DT+\mmxpow \subset \DJJ]\in
    \HshortZZflag^{\Gx \times \Gy}$ to
    the connected component
    of $[\DI+\mmxpow]\in \HshortZZ^{\Gx}$.
\end{proposition}
\begin{proof}
    The element $Q$ is a non-zero-divisor on $\DT/\DI\cdot \DT$ by
    Lemma~\ref{ref:Qregular:lem}. Moreover, $\DT_{2, b} \subset \pp$ by
    Condition~\ref{ref:framelike:def}\ref{it:framelikeSaturation}, thus every
    of $\DI + (x_1, \ldots ,x_n)^{a+1}\subset \DS$ induces a deformation
    of $\DJJ$ by keeping $Q$ and $\pp$ fixed. This gives the required section
    of $\pr_1$.
\end{proof}

Now, we prove the following abstract version of
Proposition~\ref{ref:retractionFromFrame:prop} from the introduction.
\begin{proposition}\label{ref:retractionFromFrameLike:prop}
    Let as before, $\DJJ$ be a frame-like ideal
    \begin{enumerate}
        \item\label{it:firstRetractionFrameLike} the scheme $(\HshortZZ, [\DJJ])$ locally retracts
            to $(\HshortZZ^{\Gdiag}, [\DJJ])$,
        \item\label{it:secondRetractionFrameLike} the scheme $(\HshortZZ^{\Gdiag}, [\DJJ])$ locally retracts to
            $(\HshortZZ^{\Gx \times \Gy}, [\DJJ])$,
        \item\label{it:thirdRetractionFrameLike} $(\HshortZZ^{\Gx \times \Gy}, [\DJJ])$
            locally retracts to $(\OHilb^{\Gx}(\mathbb{A}^n), [\DI + (x_1, \ldots
            ,x_n)^{a+1}])$.
    \end{enumerate}
\end{proposition}

\begin{proof}
    The existence of the local
    retractions~\eqref{it:firstRetractionFrameLike},~\eqref{it:secondRetractionFrameLike},~\eqref{it:thirdRetractionFrameLike}
    is proven respectively in
    Proposition~\ref{ref:quadriclemma:prop},
    Proposition~\ref{ref:planesOnQuadric:prop},
    Propositions~\ref{ref:JToPair:prop}-\ref{ref:IToPair:prop}.
\end{proof}

\begin{proof}[Proof of Theorem~\ref{ref:mainthm:thm}]

    Fix a singularity type $\singula$ and let $n=8$. By \cite[Proposition~4.4]{Vakil_MurphyLaw} there
    exists a field $\kk$ and a smooth general type surface $Z$ over $\kk$ and an embedding $Z \into
    \mathbb{P}^4_{\mathbb{Z}}$ such that
    $[Z]\in \OHilb(\mathbb{P}^4_{\mathbb{Z}})$ has singularity type $\singula$. Let $\DS_0 :=
    \bigoplus_i H^0(\mathbb{P}^4, \OO(i)) = \kk[x_1, \ldots ,x_{5}]$. Let
    $M\gg 0$, in particular $M > n$, and let
    $\DI_0 := I(Z)_{\geq M} \subset \DS_0$, so that $[I_0]\in
    \OHilb^{\Gmult}(\mathbb{A}^{5})$ has singularity type $\singula$,
    see~\cite[Lemma~4.1]{Haiman_Sturmfels__multigraded}.
    Let $\DS = \DS_0[x_{6}, x_{7}, x_{8}]$ be a polynomial ring over
    $\DS_0$ and let $\DI = I_0\cdot \DS$.
    Let
    \begin{equation}\label{eq:tangentToI}
        \tanspace_{I} := \Hom_{\DS}(\DI, \DS/\DI).
    \end{equation}

    Fix an action of $\Gres = \Gmult$ on $\DS$ acting with weight one on
    coordinates $x_6, x_7, x_8$ and fixing $\DS_0$. Since $\DI$ is generated by
    elements of $\DS_0$, we have $(\tanspace_I) = (\tanspace_I)_{\geq 0}$ with respect to
    the grading induced by $\Gres$.
    Let $\Hshort = \OHilb^{\Gmult}(\mathbb{A}^{8}_{\mathbb{Z}})$ and let $\theta\colon \Hplus \to
    \Hshort$ be its \BBname{} decomposition with respect to the $\Gres$-action.
    Since $\tanspace_I$ is non-negatively graded, the map $\theta$ is an open
    immersion near $[I]$. Hence, on an neighbourhood $U$ of $[I]\in \Hshort$ there is a
    retraction
    \begin{equation}\label{eq:smallretraction}
        U \to \Hshort^{\Gres}.
    \end{equation}
    But a neighbourhood of $[I]\in \Hshort^{\Gres}$ is canonically isomorphic
    to a neighbourhood of $[I_0]\in
    \OHilb^{\Gmult}(\mathbb{A}^{5})$. Hence,
    $(\OHilb^{\Gmult}(\mathbb{A}^{5}), [I_0])$ is a local retract of
    $(\OHilb^{\Gmult}(\mathbb{A}^{8}), [I])$.  We note that
    \begin{enumerate}
        \item $I_n = 0$,
        \item $\depth(\DS_{+}, \DS/\DI) \geq 3$, because $x_6, x_7, x_8$ form a regular sequence.
    \end{enumerate}
    We fix  $a\geq \reg(\DI)+1$, where $\DI$ is the regularity of ideal $\DI$.
    Then by~\cite[Proposition~3.1]{erman_Murphys_law_for_punctual_Hilb} the component of
    $H$ containing $[\DI + \mmxpow]$ is isomorphic to the component of $H$
    containing $[\DI]$ (here the choice of large enough $a$ is crucial).

    If $\charr \kk \neq 2$, then let $\DJJ$ be a \TNTframe{} for $\DI$ of size
    $a$. If $\charr \kk = 2$, then let $\DJJ$ be a tweaked frame for $\DI$. By
    Lemma~\ref{ref:abstraction:lem} or Lemma~\ref{ref:abstractionCharTwo:lem},
    the ideal $\DJJ$ is $n$-frame-like.

    By Proposition~\ref{ref:retractionFromFrameLike:prop} we obtain a retraction from an open neighbourhood of $[\DJJ]$ in
    $\HshortDR$ to a neighbourhood of $[\DI+\mmxpow]$ in $H := \OHilb^{\Gmult}\left( \mathbb{A}^8
    \right)$, which is isomorphic to a neighbourhood of $[\DI]$.
    Composing with~\eqref{eq:smallretraction}, we obtain the desired
    retraction.
\end{proof}
\begin{remark}
    Erman~\cite{erman_Murphys_law_for_punctual_Hilb} proved that
    $\coprod_{n} \OHilb^{\Gmult}(\mathbb{A}^n)$ satisfies Murphy's Law by a different
    reduction from $Z$ to $I(Z)$ using a sufficiently positive embedding of
    $Z$. His method is not
    applicable here, since the obtained $I(Z)$ is generated by quadrics and we require $I_2
    = 0$.
\end{remark}

\section{Corollaries of Theorem~\ref{ref:mainthm:thm} and
examples}\label{sec:after}

\begin{corollary}[Answer to Question~\ref{que:nonreducedness}]
    There are non-reduced points on the schemes $\HshortZZ$ and on
    $\HshortK$, where $\KK$ is any field.
\end{corollary}
\begin{proof}
    For every pointed scheme $(Y, y)$ in the singularity type of
    $[\Spec(\mathbb{Z}[u]/u^2), V(u)]$ the ring $\OO_{Y, y}$ is
    $\mathbb{Z}$-flat but not reduced, and
    hence contains an element $v$ such that $\mathbb{Z}[v]/v^2 \subset
    \OO_{Y, y}$. As in
    Theorem~\ref{ref:mainthm:thm}, suppose that $(Y, y)$ is such
    a scheme with a retraction $(X, x)\to (Y, y)$ from an open subscheme $X$
    of $\HshortZZ$. Then the
    pullback $\OO_{Y, y} \to \OO_{X,x}$ has a section, hence is an injective
    homomorphism. In particular, $\OO_{X, x}$ is non-reduced as well.
    This proves the claim for $\HshortZZ$.
    The injective homomorphisms
    \[
        \mathbb{Z}[v]/v^2 \into \OO_{Y, y} \into \OO_{X, x}
    \]
    stay injective under $(-)\tensor_{\mathbb{Z}} \mathbb{Q}$, hence the claim
    follows for $\HshortQQ$.
    To prove the claim for $\HshortFp$, we argue as above for the singularity
    $[\Spec(\mathbb{F}_p[u]/(u^2))]$. The claim for arbitrary field $\KK$ now follows
    from base change.
\end{proof}
\begin{corollary}[Answer to Question~\ref{que:charp}]\label{ref:charp:cor}
    The scheme $\HshortZZ\to \Spec(\mathbb{Z})$ has components lying entirely
    in the fiber over $\Spec(\mathbb{Z}/p)$ for all primes $p$.
\end{corollary}
\begin{proof}
    For every pointed scheme $(Y, y)$ in the singularity type of
    $[\Spec(\mathbb{Z}/p)]$ we have $p\OO_{Y, y} = 0$. As in
    Theorem~\ref{ref:mainthm:thm}, suppose that $(Y, y)$ is such
    a scheme with a retraction $(X, x)\to (Y, y)$ from an open subscheme $X$
    of $\HshortZZ$. Then the pullback map $\OO_{Y, y} \to \OO_{X, x}$ implies
    that $p\OO_{X, x} = 0$, hence each component containing $x$ lies entirely in
    characteristic $p$.
\end{proof}

We can extend Corollary~\ref{ref:charp:cor} by
considering higher infinitesimal neighbourhoods.
\begin{corollary}\label{ref:liftingPathologies:cor}
    For every prime $p$ and every ${\nu}\geq 0$ there exists a finite field $\kk$ and
    a finite irreducible $\mathbb{F}_p$-scheme $R$ with residue field $\kk$
    which lifts to $\mathbb{Z}/p^{\nu}$ but not to any ring $A$ with $p^{{\nu}}A\neq 0$.
    In particular $R$ does not lift to $\mathbb{Z}/p^{{\nu}+1}$.
\end{corollary}
\begin{proof}
    Using Theorem~\ref{ref:mainthm:thm}, choose a member $(Y, y)$ in the
    singularity type of $\singula = [\Spec(\mathbb{Z}/p^{\nu})]$ such that there exists a
    retraction $(X, x)\to (Y, y)$ from an open subscheme $X$ of $\HshortZZ$.
    Let $\kk := \kappa(x) = \kappa(y)$ and $[R] := x$. Since $(Y, y)$ in the
    type of $\singula$ in the
    smooth equivalence relation, we have $p^{\nu}\OO_{Y, y} = 0$ and a morphism
    $\varphi\colon \Spec(\mathbb{Z}/p^{\nu}) \to Y$. Composing $\varphi$ with the
    section of $X \to Y$, we get a lifting of $R$ to $\Spec(\mathbb{Z}/p^{\nu})$.
    The scheme $R$ embeds into $\mathbb{A}^{16}_{\kk}$, hence its lifting over
    $A$ would embed into $\mathbb{A}^{16}_{A}$ and give a morphism $\Spec(A)
    \to \HshortZZ$, which (after perhaps localizing $A$) restricts to
    $\Spec(A)\to X$. Hence we obtain $\Spec(A) \to Y$ and so $p^{{\nu}}A = 0$.
\end{proof}

So far our arguments built upon Vakil's
construction, which in turn depends on Mn\"ev-Sturmfels
universality for incidence schemes~\cite{Vakil__Mnev_Sturmfels_universality} and on results about abelian
covers~\cite[\S4]{Vakil_MurphyLaw}.
The Mn\"ev-Sturmfels construction requires $\mathbb{P}^2_{\kk}$ to have enough
$\kk$-points, hence usually it does not work over $\kk = \mathbb{F}_p$ (this is the
reason why in Corollary~\ref{ref:liftingPathologies:cor} we do not obtain
algebras with residue field $\mathbb{F}_p$). The theory of abelian covers,
while in principle constructive, is not very prone to become explicit either.

In this final part we explicitly construct
appropriate points of the Hilbert scheme by hand,
bypassing Vakil's work, for several singularity types. First, we
note that one can obtain explicit examples of non-reduced
points on $\HshortZZ$ by taking a \TNTframe{} for the truncation of the cone
over a curve from Mumford's famous example~\cite[\S13]{HarDeform} or the
examples of Martin-Deschamps and Perrin~\cite{MartinDeschamps_Perrin}.
We give an explicit example by framing a degree $3$, genus $-2$ reducible curve
from~\cite[Prop~0.6]{MartinDeschamps_Perrin}.
\begin{example}
    Let $\kk$ be of characteristic zero. Let $S = \kk[x_1, \ldots , x_7, y_1,
    \ldots ,y_7]$ be a polynomial ring and $K = (x_1^2,\ x_1x_2,\ x_2^2(x_3 +
    x_4),\ x_1x_4^3 + x_2(x_3 + x_4)x_3^2)$. The regularity of $K$ is four.
    Let $\DI = K\cap (x_1, \ldots ,x_4)^4$ and let $J$ be a \TNTframe{} for
    $\DI$ with $a = 5$. Then $[J]\in \Hshortkk$ is non-reduced.
\end{example}

Below, we
give explicit components of
$\OHilb_{\pts}(\mathbb{A}^{6}_{\mathbb{Z}})$ lying in characteristic $p$ for small $p$; in fact we give
$\mathbb{F}_p$-points of these components. The proof is obtained by
replacing the construction of Theorem~\ref{ref:mainthm:thm} by some explicit
computations. Let $\kk = \mathbb{F}_p$ and let $\DR \subset \mathbb{A}^n_{\kk}$ be a finite
scheme given by a homogeneous ideal. The examples below employ the
following line of argument:
\begin{enumerate}
    \item\label{it:homogeneous} check that $\dim(\GL_{n}\cdot [\DR]) =
        \dim_{\kk} T_{\Hshortkk^{\Gmult}, [\DR]}$,
    \item conclude that $[\DR]\in \Hshortkk^{\Gmult}$ is smooth,
    \item\label{it:W2} verify that $\DR$ does not lift to $W_2(\kk)$.
    \item use~\ref{it:homogeneous}-\ref{it:W2} to conclude that the
        component of $\HshortZZ^{\Gmult}$ containing $[\DR]$ lies entirely over
        $\kk$. This is a known argument, see
        e.g.~\cite[Lemma~5.7]{Ekedahl_On_nonliftable_CY_threefolds}. If this
        holds, also a neighbourhood of $[\DR]\in
        \HplusZZ$ lies over $\kk$.
    \item\label{it:TNT} check that $\DR$ has TNT and conclude that
        $\theta\colon\HplusZZ \to
        \HshortZZ$ is an open immersion in a neighbourhood of $[\DR]$.
\end{enumerate}
The heart of all examples is the observation that the ideal
\[
    K = (x_1x_2 + x_3x_4 + x_5x_6) + (x_2^p, x_4^p, x_6^p)
\]
satisfies Properties~\eqref{it:homogeneous}-\eqref{it:W2}. It remains to reduce $K$ to
dimension zero so as not to lose these properties and additionally gain TNT.
We present one such reduction below.

\begin{example}
    Let $q = p^e$ be a prime power, let $\kk = \mathbb{F}_p$ as before and consider the ideal $I' = (x_1x_2 +
    x_3x_4 + x_5x_6) + (x_2^q, x_4^q, x_6^q) \subset \DS = \kk[x_1, \ldots
    , x_6]$. Let $I$ be its saturation and let
    $\DJJ = I + (x_1, x_3, x_5)^{q+1}$.
    Below all unjustified claims are checked with \texttt{Macaulay2} for $q =
    3, 4, 5$. Hence, we obtain examples in characteristics $\leq 5$.
    First, the stabilizer of $\DJJ$ is $10$-dimensional, given by
    \begin{equation}\label{eq:stabilizerI}
        \begin{pmatrix}
            \lambda D^{-1} & 0\\
            0 & D^T
        \end{pmatrix},
    \end{equation}
    where $\lambda\in \kk^*$ and  $D\in \GL_{3}$.
    Hence the $\GL_6$-orbit is $26$-dimensional.

    \newcommand{\DSWtwo}{\widetilde{S}}%
    \newcommand{\DJWtwo}{\widetilde{J}}%
    To prove that $\DJJ$ does not lift to $W_2(\kk)$, we argue similarly as
    in~\cite[Proposition~6.1.1]{Zdanowicz_liftability_mod_p2}.
    Let $\DSWtwo := W_2(\kk)[X_1, \ldots ,X_6]$. Suppose that $\DR$ lifts to
    $W_2(\kk)$. Then there exists an ideal $\DJWtwo \subset \DSWtwo$ such that
    $\DSWtwo/\DJWtwo$ is an embedded deformation of $\DS/\DJJ$ over
    $W_2(\kk)$.
    In particular, the syzygy
    \[
        (x_1x_2 + x_3x_4 + x_5x_6)^3 = (x_1)^3\cdot x_2^3 + (x_3)^3\cdot x_4^3
        + (x_5)^3\cdot x_6^3
    \]
    lifts to a syzygy between generators of $\DJWtwo$, which means that
    \begin{equation}\label{eq:syzygy}
        (X_1X_2 + X_3X_4 + X_5X_6)^3 - X_1^3 X_2^3 - X_3^3 X_4^3
        - X_5^3X_6^3 \in p\DJWtwo.
    \end{equation}
    We have $p\DSWtwo  \simeq \DS$ as $\DSWtwo$-modules and $p\DJWtwo \simeq
    \DJJ$ in this isomorphism. Equation~\eqref{eq:syzygy} translates into
    \begin{equation}\label{eq:syzygytwo}
        2x_1x_2x_3x_4x_5x_6 + (x_1x_2)^2(x_3x_4+x_5x_6) + (x_3x_4)^2(x_1x_2 +
        x_5x_6) + (x_5x_6)^2(x_1x_2 + x_3x_4)\in \DJJ,
    \end{equation}
    But $(x_1x_2)^2(x_3x_4+x_5x_6) \equiv (x_1x_2)^3 \equiv 0 \mod \DJJ$,
    hence~\eqref{eq:syzygytwo} is equivalent to $2x_1x_2x_3x_4x_5x_6\in \DJJ$,
    which is false. Hence, we obtain a contradiction (the argument for $q = 4$
    should be ramified here).
    Finally, we check directly that for $\tanspace := \Hom_{\DS}(\DJJ, \DS/\DJJ)$ we
    have $\dim_{\kk} (\tanspace)_0 = 26$, which verifies Property~\ref{it:homogeneous} and that
    $\dim_{\kk} (\tanspace)_{<0} = 6$, which verifies Property~\ref{it:TNT}.
\end{example}

\end{document}